\documentclass[reqno,10.5pt]{amsart}
\usepackage{a4,latexsym,amssymb,amsfonts}

\usepackage{amsmath,amscd} 
\usepackage{mathrsfs}

\usepackage{pdfsync,verbatim} 

\usepackage{enumerate}
\usepackage{amsthm}
\usepackage[english]{babel}

\usepackage{color}
\usepackage{cite}

\usepackage{graphicx}

\def\cA{{\mathcal A}}

\def\cD{{\mathcal D}}

\def\cF{{\mathcal F}}
\def\cH{{\mathcal H}}

\def\cN{{\mathcal N}}

\def\cP{{\mathcal P}}

\def\cS{{\mathcal S}}
\def\cU{{\mathcal U}}
\def\cW{{\mathcal W}}

\def\bH{{\mathbf H}}

\def\bU{{\mathbf U}}
\def\bZ{{\mathbf Z}}

\newcommand{\bbC}{{\mathbb{C}}}
\newcommand{\bbR}{{\mathbb{R}}}

\newcommand{\bbZ}{{\mathbb{Z}}}

\def\Re{\operatorname{Re}}
\def\Im{\operatorname{Im}}
\def\Res{\operatorname{Res}}
\def\arctanh{\operatorname{arctanh}}

\def\la{\langle}
\def\ra{\rangle}
\def\p{\partial}

\def\eps{\varepsilon}

\def\z{\zeta}

\def\raw{\longrightarrow}

\def\sW{\mathcal{W}}

 \def\HollowBox #1#2{{\dimen0=#1 \advance\dimen0 by -#2       
       \dimen1=#1 \advance\dimen1 by #2                       
        \vrule height #1 depth #2 width #2                    
        \vrule height 0pt depth #2 width #1                   
        \llap{\vrule height #1 depth -\dimen0 width \dimen1}% 
       \hskip -#2                                             
       \vrule height #1 depth #2 width #2}}

\newtheorem{Thm}{Theorem}
\newtheorem{thm}{Theorem}[section]
\newtheorem{prop}[thm]{Proposition}
\newtheorem{cor}[thm]{Corollary}
\newtheorem{lem}[thm]{Lemma}

\begin{document}

\title[Bergman kernel on the unbounded worm]{Bergman kernel and projection on the unbounded
  Diederich--Forn\ae ss worm domain} 

\author[S. G. Krantz]{Steven G. Krantz}
\address{Campus Box 1146
Washington University in St. Louis
St. Louis, Missouri 63130}
\author[M. M. Peloso]{Marco M. Peloso}
\address{Dipartimento di Matematica ``F. Enriques''\\
Universit\`a degli Studi di Milano\\
Via C. Saldini 50\\
I-20133 Milano}
\author[C. Stoppato]{Caterina Stoppato}
\address{Istituto Nazionale di Alta Matematica\\
Unit\`a di Ricerca di Firenze c/o
DiMaI ``U. Dini'' Universit\`a di Firenze\\
Viale Morgagni 67/A\\
I-50134 Firenze\bigskip}

\email{sk@math.wustl.edu}
\email{marco.peloso@unimi.it}
\email{stoppato@math.unifi.it}
\thanks{Second author supported in part by the 2010-11 PRIN grant
  \emph{Real and Complex Manifolds: Geometry, Topology and Harmonic Analysis}  
  of the Italian Ministry of Education (MIUR)}
\thanks{Third author supported by the FIRB grant \emph{Differential
    Geometry and Geometric Function Theory} of the MIUR} 

\date{\today}

\begin{abstract}
In this paper we study the Bergman kernel and projection on the unbounded
 worm domain
$$
\cW_\infty = \big\{ (z_1,z_2)\in\bbC^2:\,
\big|z_1-e^{i\log|z_2|^2}\big|^2<1 ,\ z_2\neq0\big\} \, .
$$
We first show that the Bergman space of $\cW_\infty$ is infinite dimensional.
Then we study Bergman kernel $K$ and Bergman projection
$\cP$ for $\cW_\infty$.
 We prove that $K(z,w)$ extends holomorphically  in $z$ (and 
 antiholomorphically in $w$) near each point of the boundary except for a 
 specific subset that we study in detail.
By means of an appropriate asymptotic expansion for $K$, we prove that 
the Bergman projection $\cP:W^s\not\to W^s$ if $s>0$ and 
$\cP:L^p\not\to L^p$ if $p\neq2$, where $W^s$
denotes the classic Sobolev space, and $L^p$ the Lebesgue space,
respectively,  on $\cW_\infty$.
\end{abstract}

\maketitle

\section*{Introduction}\label{intro}
\vskip12pt

In this paper we study the Bergman kernel and projection on the unbounded
 domain
\begin{equation}\label{unbdd-worm}
\cW_\infty = \big\{ (z_1,z_2)\in\bbC^2:\,
\big|z_1-e^{i\log|z_2|^2}\big|^2<1 ,\ z_2\neq0\big\}
\end{equation}
(see Figure 1).
%%%%%%%%%%%%%%%%%%%%%%%%%
\begin{figure}[t]
  \begin{center}
  %\fbox{
  \includegraphics[height=8cm]{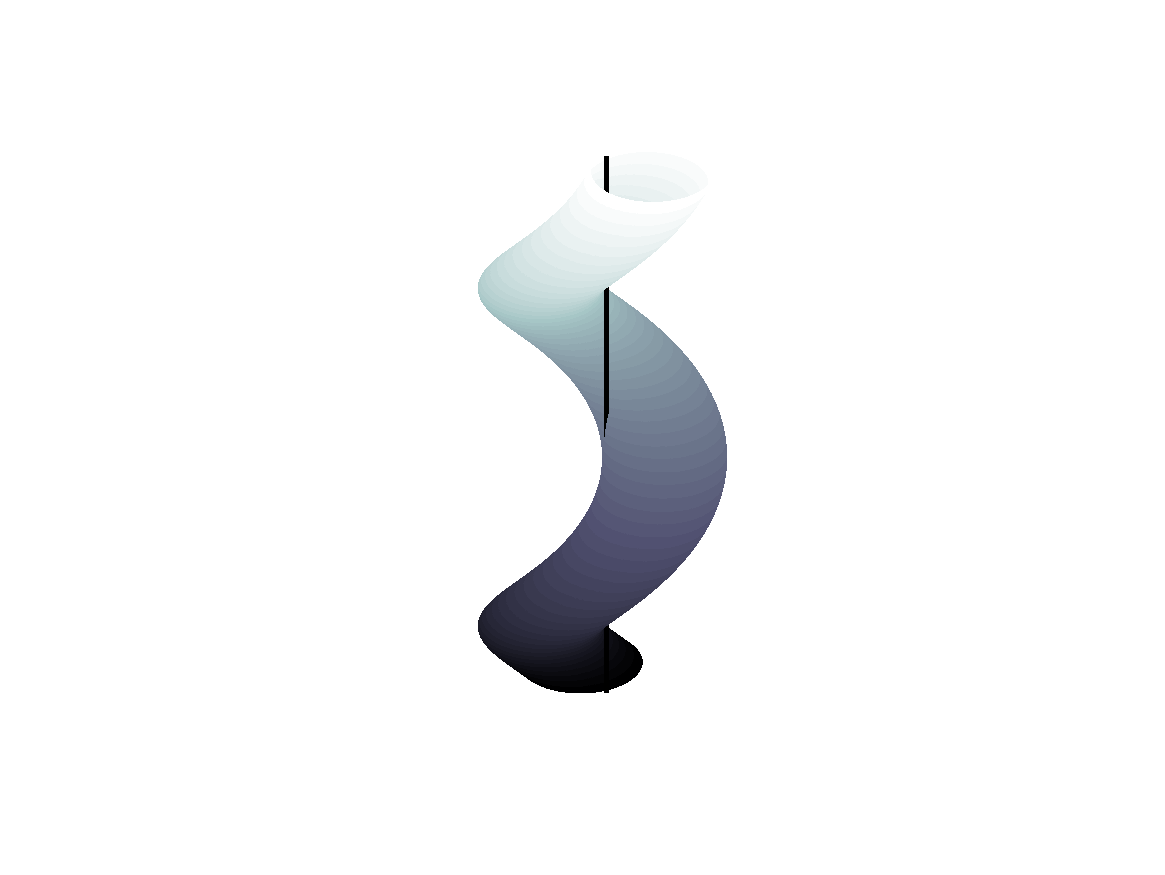}%}
\end{center}
  \caption{\footnotesize  A portrait in $\bbC \times \bbR$ of a
    section of $\cW$ . The first variable $z_1$ spans in the
    horizontal plane $\bbC$, while $\log|z_2|^2$ spans along the
    vertical line $\bbR$ (drawn in black).} 
\end{figure}
%%%%%%%%%%%%%%%%%%%%%%%%%
Recall that, for $\mu>0$, the Diederich--Forn\ae ss worm domain
$\cW_\mu$ is defined by
\begin{equation}\label{bdd-worm}
\cW_\mu = \big\{ (z_1,z_2)\in\bbC^2:\,
\big|z_1-e^{i\log|z_2|^2}\big|<1-\eta\big(\log|z_2|^2\big)\big\} \, ,
\end{equation}
where $\eta$ is a smooth, even, convex, non-negative function on the real 
line, chosen so that $\eta^{-1}(0)=[-\mu,\mu]$ and so that $\cW_\mu$ is
bounded, smooth, and pseudoconvex.  Its boundary is strongly 
pseudoconvex except at the points $\big\{ (0,z_2):\, \big|\log |z_2|^2\big|\le
\mu\big\}$.  The worm domain $\cW_\mu$ was introduced in \cite{DFo1} by
K. Diederich and J. E. Forn\ae ss and
turned out to be of great interest as it
provides ({\em counter-})examples for many important phenomena.   

Diederich and Forn\ae ss showed that 
the worm
is the first example of
a smoothly bounded domain with nontrivial Nebenh\"{u}lle.
Moreover, it gives an example of a smoothly
bounded, pseudoconvex domain which lacks a global
plurisubharmonic defining function.
Furthermore, nearly 15 years after its introduction, the worm domain showed 
another feature that is of great interest.
In order to describe this feature of $\cW_\mu$ 
and to motivate our present work on 
$\cW_\infty = \bigcup_{\mu>0} \cW_\mu$, let us first recall some preliminary 
material concerning the Bergman space of a complex domain and the associated 
Bergman projection, as well as its role in the study of the geometry of the domain.

\bigskip

If $\Omega$ is a given domain in $\bbC^n$, denote by
$A^2(\Omega)$ the space of 
holomorphic functions on $\Omega$ that are square integrable with respect to
Lebesgue measure.   Then,
$A^2(\Omega)$ is a closed subspace of $L^2(\Omega)$ and the Hilbert
space projection 
$$
P: L^2(\Omega) \raw A^2(\Omega)
$$
can be represented by an integration formula
$$
P f(z) = \int_\Omega K(z,\zeta) f(\zeta) \, dV(\zeta) \, .
$$
The kernel $K(z,\zeta) = K_\Omega(z,\zeta)$ is called the {\it Bergman
  kernel}.  There exists a vast literature on the Bergman kernel and
projection, and their role in geometric analysis in one and several
variables; here we only mention\cite{ChSh}, \cite{Kr1} and
\cite{Straube-lectures} for
the basic ideas and a general overview.

Clearly the Bergman projection $P$ is bounded on $L^2(\Omega)$. 
Its regularity, or irregularity, in other norms or more general topologies
is of great interest.

When $\Omega$ is assumed
to be smooth, bounded and pseudoconvex, S. R. Bell \cite{Bel1} 
formulated the notion of 
{\it Condition $R$}, that is the requirement that 
$P: C^\infty(\overline{\Omega}) \to
C^\infty(\overline{\Omega})$ is bounded.  
The work of Bell and that of Bell and E. Ligocka \cite{BelLi} led to the following 
fundamental result: if $\Phi: \Omega_1 \raw \Omega_2$ is a 
biholomorphic mapping between smoothly bounded, 
pseudoconvex domains of $\bbC^n$, one of which satisfies Condition 
$R$, then  $\Phi$ extends to be a $C^\infty$
diffeomorphism of $\overline{\Omega}_1$ to $\overline{\Omega}_2$.

Many different classes of domains are known to satisfy
Condition $R$: e.g., strongly pseudoconvex 
domains and domains of finite type, domains with 
real-analytic 
boundary, complete Hartogs domains in $\bbC^2$, 
domains that admit a defining function that is plurisubharmonic on
the boundary, see 
\cite{Cat1}, \cite{Cat2}, \cite{DFo2}, \cite{BoStraube} and
\cite{Bo-St-pluri}, respectively. 
On the other hand, considerable 
effort has been put into the search for examples of domains that 
do not satisfy Condition $R$. Among the first works on this matter we 
might mention \cite{Ba1}, where D. Barrett  showed that
there exists a smoothly bounded, non-pseudoconvex  
domain $\Omega$ in $\bbC^2$ on which Condition $R$ fails.
In particular, Barrett's work provides some insight on the problem 
caused by rapidly varying normals to the boundary; see also\cite{Ba2}.

Clearly, one way to try to measure whether a domain $\Omega$ 
satisfy or not Condition $R$ is to determine the Sobolev regularity of 
$P$; namely, whether or not, for $s > 0$, the projection $P$ preserves 
the Sobolev space $W^s(\Omega)$ (see, e.g.  \cite{Hor1}, \cite{Kr2}).   
In this direction, J. J. Kohn \cite{Kohn-quantitative} and B. Berndtsson 
and P. Charpentier \cite{Be-Ch} proved (independently and with 
completely different approaches) that for each smooth
bounded pseudoconvex domain $\Omega$ in $\bbC^n$ there exists
$s_\Omega>0$ such that $P: W^s(\Omega) \to W^s(\Omega)$ is 
bounded for $0<s<s_\Omega$.  
In \cite{Be-Ch} is it shown that $s_\Omega\ge \text{DF}(\Omega)/2$, 
where $\text{DF}$ denotes the Diederich-Forn\ae ss exponent 
of the given domain $\Omega$
\begin{multline}\label{DF-def}
\text{DF}(\Omega)= 
\sup\big\{\, 0<\delta\le1:\, \exists\text{ defining function } 
\varrho \text{ for } \Omega,\, -(-\varrho)^\delta\text{\
  plurisubharmonic on\ } \p\Omega \big\}\, .
\end{multline}
The lower bound obtained in \cite{Kohn-quantitative} is not explicit; 
one way to obtain such a lower bound is described in \cite{Pi-Za}.

An alternative method to establish regularity  is
via the Neumann operator $\cN$, that is,
the solution operator of the complex Laplacian
$\Box=\bar\partial\bar\p^*+\bar\p^*\bar\p$ 
on square-integrable $(0, 1)$-forms. In fact 
H. P. Boas and E. J. Straube \cite{BoS2} established
a connection between regularity of $\cN$ and $P$; see also 
\cite{Straube-lectures} and the references therein.

Another interesting result in this context is \cite{He-Mc-St} where
A.-K. Herbig, J. D. McNeal and Straube address the 
problem of studying on which
subspace of $C^\infty(\overline{\Omega})$ the Bergman projection 
is bounded as a map into $C^\infty(\overline{\Omega})$.

\bigskip

Consider now the worm domain $\cW_\mu$. 
Let
$\cP_\mu$ denote the Bergman projection on $\cW_\mu$ and set
$\nu=\pi/(2\mu)$. 
Boas and Straube \cite{Bo-St-Hartogs} showed that
the Bergman projection on $\cW_\mu$ maps $W^k$ into itself if 
$k$ is an integer and $k\ge\nu$, or if $k=\frac12$.
Furthermore, the result of \cite{Be-Ch} 
applies to $\cW_\mu$ so that $W^s$ must be preserved by 
$\cP_\mu$ for all $s < \text{DF}(\cW_\mu)/2$. We point out, though, 
that in \cite{DFo1} Diederich and Forn\ae ss showed 
that $\text{DF}(\cW_\mu)\le\nu$, 
(see also \cite{KrPe} for details).
\medskip

In the direction of understanding {\em irregularity} of the Bergman
projection, 
it was C. O. Kiselman \cite{Ki} who established an important
connection between the worm domain and Condition $R$.  He 
proved that, for a certain non-smooth version of the worm, a 
form of Condition $R$ fails.

 Stemming from the ideas developed in \cite{Ki}, 
in \cite{Barrett-Acta} Barrett proved the 
ground-breaking fact that
\begin{itemize}
\item[(i)]
$\cP_\mu: W^s(\cW_\mu)\not\to W^s(\cW_\mu)$ when $s\ge \nu$;\smallskip
\end{itemize}
where $W^s(\cW_\mu)$ denotes the standard Sobolev space.  By the same
proof, see also \cite{KrPe}, it also follows that
\begin{itemize}
\item[(ii)]
$\cP_\mu: L^p(\cW_\mu)\not\to L^p(\cW_\mu)$
for $\big|\frac1p-\frac12\big|\ge \nu/2$.
\medskip
\end{itemize}
Based on Barrett's result on the irregularity of $\cP_\mu$, the work of M. 
Christ \cite{Chr} showed that the worm domain is a counterexample to
Condition $R$. 
 After decomposing the space of square-integrable $(0,1)$-forms as 
$L^2_{(0,1)}(\cW_\mu) = \oplus_{j \in \bZ}\cH_j^1$, where
$\cH^1_j
= \big\{ u \in L^2_{(0,1)}(\cW_\mu) : \ u(w_1,e^{i\theta}w_2) =
e^{ij\theta} u(w_1,w_2)\big\}$, 
he showed that for all $s>0$ (apart from a discrete set of
exceptions) the Neumann operator $\cN$ satisfies an a priori estimate 
$||\cN u||_{W^s} \leq C_{s,j} ||u||_{W^s}$ 
valid for every $u \in \cH_j^1 \cap C^\infty(\overline{\cW}_\mu)$ 
such that $\cN u \in C^\infty(\overline{\cW}_\mu)$.
If $\cN:
C^\infty(\overline{\sW}_\mu)\to
C^\infty(\overline{\sW}_\mu)$ were bounded, such estimates 
would 
contradict the irregularity of $\cP_\mu$.

\bigskip

The peculiar properties of the worm domain $\cW_\mu$ have 
already earned it considerable attention as a counterexample to 
many important phenomena  and they motivate a deeper 
study of the Bergman space of $\cW_\mu$. 
This study is extremely challenging: for instance, writing down a 
basis or even a complete system for 
$A^2(\cW_\mu)$ is still an open problem. As a step towards the 
study of $\cW_\mu$, in this paper we study the unbounded 
worm domain $\cW_\infty$ defined in \eqref{unbdd-worm}, which 
can be thought of as the limit of the smoothly bounded worm 
domains $\cW_\mu$ as  $\mu\to+\infty$. This makes it an easier 
domain to study than the original $\cW_\mu$, as we are about to see.
We will explain in our Concluding Remarks how the technique applied 
here may shed some light on the study of the original smoothly 
bounded worm domains.
For simplicity of notation, we are going to write $\cW$ 
instead of $\cW_\infty$  and $\cP$ for $\cP_\infty$ in
the remainder of this paper. 
\medskip 
  
The domain $\cW$ is clearly unbounded. Denote by $\p\cW$
its boundary.  It is well-known  (see \cite{FrGr}, \cite{ChSh}, 
and the next section for details) that 
\begin{itemize}
\item $\cW$ is pseudoconvex; 
\item $\p\cW$ is smooth except at the points 
$\cN:=\{ (z_1,0):\, |z_1|\le2\}$;
\item $\cW$ has nontrivial Nebenh\"ulle;
\item the smooth part of $\p\cW$ is strongly pseudoconvex except 
at the points of the critical annulus $\cA:=\{ 0\}\times\bbC^*$.
\end{itemize}
Here, and in what follows, $\bbC^*=\bbC\setminus\{0\}$. \medskip 

In this work we first show that the Bergman space of $\cW$ is not
trivial, showing in particular that it is infinite dimensional. 
Then we consider a biholomorphically equivalent domain 
$\cU$ that we call the {\em unwound} worm, which is also 
unbounded, but has the property that the fibers in the second 
component, that is the sets
$\{ z_2\in\bbC:\, (z_1,z_2)\in \cW\}$,
are connected. This allows us to reduce our study to a family of 
weighted Bergman spaces $\{A^2(\bU, \alpha_j)\}_{j \in \bbZ}$ on 
the upper half-plane $\bU$ and to the corresponding kernels 
$\{K_j\}_{j \in \bbZ}$. At each point of $\bU \times \bU$, we 
compute the value of $K_j$  as $\widehat \phi_\lambda(j+1)$, 
where: $\lambda$ is a number in the right half-plane 
$\bH$, associated to the given point of $\bU \times \bU$; and 
$\widehat \phi_\lambda$ denotes the Fourier transform of the 
function
\begin{equation*}
\phi_\lambda(s) = \frac{1}{2\pi^3} \frac{1}{\cosh^2 s}\Big[
\big(2\log(\cosh s)+\lambda\big)^{-2} + 4 \big( 2\log(\cosh
s)+\lambda\big)^{-3}
\Big] \, .
\end{equation*}
Altogether, we express the Bergman kernel $K$ of $\cW$ as a 
series of functions, each of which is explicitly computed in 
terms of the aforementioned $K_j$.

By means of this machinery, we prove that $K(z,w)$ 
extends holomorphically in $z$ (and antiholomorphically in $w$) 
near each point of the boundary except for a specific subset, which 
includes the critical set $(\cA\times \cW) \cup (\cW \times \cA)$. 
We then find an asymptotic expansion for $K$ near the critical set 
that allows us to prove that 

\begin{Thm}\label{irregularity}	  
For all $s>0$, the Bergman projection $\cP$ does not map the
Sobolev space $W^s(\cW)$ into itself; nor does it map
$L^p(\cW)$ into itself for any $p$ other than $2$.  
\end{Thm}

We point out again that the domain is unbounded and non-smooth.  
However, the analysis of the singularities of the Bergman kernel 
shows that the irregularity of the projection is caused by the 
pathological behavior of $K(\cdot,w)$ near each point of the 
critical annulus $\cA$, where the boundary of the domain is smooth.
\bigskip

\noindent {\em Acknowledgements}

Part of the work was done while the second author held a visiting 
position at the Department of Mathematical Sciences of the 
University of Arkansas.  He wishes to thank the Faculty and the 
Staff of the Department for the opportunity to work in such a 
stimulating environment.

Preliminary research was done when the third author was a 
fellow at the Department of Mathematics of the University of 
Milan, supported by FSE and by Regione Lombardia. She warmly 
thanks all of these institutions for the remarkable research 
opportunity. 
During part of the same period, she acknowledges partial suppport 
by the PRIN grant 
\emph{Real and Complex Manifolds: Geometry, Topology and Harmonic Analysis}  
  of the MIUR.
She is also grateful to the Department of Mathematical 
Sciences of the University of Arkansas, which she visited during the 
preparation of this paper.

It is a pleasure to thank Harold Boas for useful conversations 
about the Bergman projection.

\medskip

\section{Basic facts about $\cW$ and 
  $\cU$}\label{basic-sec} 
\medskip

We begin with the following well-known result,---see
e.g. \cite{FrGr}.
\begin{prop}\label{basics} \sl
The domain $\cW$ is pseudoconvex and has nontrivial Nebenh\"ulle.
Moreover, the boundary   
 $\p\cW$ is smooth except at the points $\cN=\{ (z_1,0):\,
 |z_1|\le2\}$ and
the smooth part of $\p\cW$ is strongly pseudoconvex except at
  the points of the critical annulus $\cA=\{ 0\}\times\bbC^*$.
\end{prop}

We write $\Delta(\z,r)$ to denote the disk of center $\z$ and radius
$r$  in $\bbC$ and $\bH$ to denote the right half-plane
in the complex plane. 
Observe that
\begin{align*}
\cW = \bigcup_{z_2 \in \bbC^*} \Delta(e^{i\log|z_2|^2},1)
\times\{z_2\}\, . 
\end{align*}
In particular, the projection of $\cW$ onto the first variable is 
$\Delta(0,2) \setminus \{0\}$.

We denote by $\log\z$ the principal branch of logarithm for
$\z\in\bbC\setminus(-\infty,0]$ and use it to define some useful 
functions on $\cW$.

\begin{lem}  \sl
Setting
\begin{equation}\label{def-L}
L(z) = \log\big(z_1e^{-i\log|z_2|^2}\big)+i\log|z_2|^2
\end{equation}
defines a complex-valued holomorphic function in the variable
$z=(z_1,z_2)\in \bbC^2$ on the domain $\cD=
\bigcup_{z_2 \in \bbC^*} \{e^{i\log|z_2|^2}\bH\} \times \{z_2\} \subset
\bbC^2$. The same is true for 
\begin{equation}\label{def-E}
E_\eta(z) := e^{\eta L(z)} = \big(z_1e^{-i\log|z_2|^2}\big)^\eta e^{i\eta\log|z_2|^2}
\end{equation}
for each $\eta \in \bbC$.
\end{lem}

\proof 
It is elementary to check that $L(z)$ is well defined on
$\cD\supseteq\cW$ and that it is annihilated by $\overline{\p}$.
\qed
\medskip \\

We point out that the fiber of $\cW$ over each $z_1 \in
\Delta(0,2)\setminus\{0\}$ is not connected and that $L(z)$ is locally constant
in $z_2$, but not constant. The same happens with $E_\eta(z) $ for $\eta
\in \bbC \setminus \bbZ$ (while $E_k(z) = z_1^k$ for all $k \in \bbZ$, $z \in
\cW$). 
\medskip \\

We can next explicitly construct elements of the Bergman
space $A^2(\cW)$,  proving in particular that it is 
infinite dimensional.

\begin{prop} \label{nontrivial}	   \sl
Let $\mu \in (0,+\infty), \eta \in \bbC, c>\log 2, j \in \bbZ, m \in \bbR$.
Then:
\begin{itemize}
\item[(i)]
the function $E_\eta(z) z_2^j$ belongs to $A^2(\cW_\mu)$ if and only if 
$\Re \eta>-1$; 
\item[(ii)] the function
$$
F_{\eta,c,j,m}(z) =
\frac{E_{\eta}(z) z_2^j}{\big(L(z) - c\big)^m}
$$
belongs to $A^2(\cW_\mu)$ if and only if $\Re \eta> -1$, for any
$m\in\bbR$, or $\Re \eta= -1$, for $m>1$. 
\end{itemize}

Finally, 
\begin{itemize}
\item[(iii)]
if $\Re\eta> -1, \Im\eta = \frac{j+1}2$ and 
$m>\frac12$, then $F_{\eta,c,j,m}\in A^2(\cW)$, and if
$\eta=-1+i\frac{j+1}2$ and $m>1$, then $F_{\eta,c,j,m}\in A^2(\cW)$.
\end{itemize}
\end{prop}

\proof
We write $dV$ to denote the Lebesgue measure both in $\bbC$ and in
$\bbC^2$ and $\arg\z$ to denote the principal branch of the argument
of $\z\in\bbC\setminus(-\infty,0]$.  
We have
\begin{align*}
&\|F_{a+ib,c,j,m}\|_{A^2(\cW_\mu)}^2 
=\int_{\cW_\mu}\Big\vert\frac{E_{a+ib}(z) z_2^j}{\big(L(z) - c\big)^m}\Big\vert^2\, dV(z)\\ 
&= \int_{-\mu< \log|z_2|^2<\mu}\int_{\Delta(e^{i\log|z_2|^2},1)}
\frac{|z_1|^{2a} |z_2|^{2j}
\exp \big\{-2b[\arg(z_1e^{-i\log|z_2|^2}) +
 \log|z_2|^2]\big\}  }{ 
\big[ (\log|z_1| - c)^2+\big(\arg(z_1e^{-i\log|z_2|^2}) + \log|z_2|^2\big)^2\big]^m}
\, dV(z_1)dV(z_2)  \\ 
&= \int_{-\mu< \log|z_2|^2<\mu}\int_{\Delta(1,1)}
\frac{|\zeta|^{2a}|z_2|^{2j}\exp\big\{-2b(\arg(\zeta) + \log|z_2|^2)\big\}}{
\big[ (\log|\zeta| -c)^2+ (\arg(\zeta) + \log|z_2|^2)^2 \big]^m} \, dV(\zeta)
\,dV(z_2)\\ 
&= 2\pi \int_{e^{-\mu/2}}^{e^{\mu/2}}\int_0^{\frac \pi 2}\int_0^{2\cos\theta}
\frac{r^{2a+1} \rho^{2j+1}  e^{-2b(\theta+ \log \rho^2)}}{ \big[(\log r - c)^{2}
+(\theta+ \log \rho^2)^2\big]^m}\, dr\, d\theta\,
 {d\rho}\\ 
&= \pi \int_0^{\frac \pi 2}\int_{\theta-\mu}^{\theta+\mu}\int_{-\infty}^{\log(2\cos\theta)}
\frac{e^{2(a+1)s} e^{(t-\theta)(j+1)}e^{-2bt}}{ \big[(s-c)^{2}+t^2\big]^m}\, ds\,
dt \,d\theta \\ 
&= \pi \int_0^{\frac \pi 2}\int_{\theta-\mu}^{\theta+\mu}\int_{-\infty}^{\log(2\cos\theta)}
\frac{e^{2(a+1)s} ds}{ \big[(s-c)^{2}+t^2\big]^m}\,
e^{t(j+1-2b)}\,dt\, e^{-\theta(j+1)}\,d\theta\,. 
\end{align*}
For $\mu \in (0, +\infty)$, the above integral converges if 
 and only if 
$$
\int_0^{\frac \pi 2}\int_{\theta-\mu}^{\theta+\mu}\int_{-\infty}^{\log(2\cos\theta)}
\frac{e^{2(a+1)s} }{ \big[(s-c)^{2}+t^2\big]^m}\, ds\,dt\, d\theta
$$
is finite, that is, if and only if
$$
\int_{\frac\pi2-\mu}^{\frac\pi2+\mu}\int_{-\infty}^{0}
\frac{e^{2(a+1)s} }{ \big[s^2 +\eps^2+t^2\big]^m}\, ds\,dt
$$
is finite, where $\eps=c-\log2>0$. 
Now,   assertions (i) and (ii) follow at once.

Next, if $\mu$ is taken to be $+\infty$ and $b= \frac{j+1}2$, we have
\begin{align*}
\|F_{a+ib,c,j,m}\|_{A^2(\cW)}^2 
%\leq C\int_{\bbR}\int_{-\infty}^{\log2} \frac{1}{ \big[(s-c)^{2}+t^2\big]^m}\,ds\,dt\,,\\
&\leq C\int_{-\infty}^{\log 2}\int_\bbR
\frac{e^{2(a+1)s}}{\big[(s-c)^{2}+t^2\big]^m}\, dt\, ds
%\\ &= C \int_{-\infty}^{\log 2} \frac{1}{c-s}\, ds < +\infty\, ,
\end{align*}
and again (iii) follows easily.
\qed
\medskip

In order to study the Bergman space it is convenient to ``unwind'' the
domain $\cW$ as follows.

\begin{prop}\label{prop-U}
For $z = (z_1,z_2) \in \cW$ set
\begin{equation}\label{Phi-def} 
\Phi(z) = \big(-i (L(z)-\log2) , z_2\big)\, .
\end{equation}
Moreover, let
\begin{equation}\label{U-def}
\cU = \Big\{(u+iv,w_2) \in \bbC^2: v>0,\,  \big|u-\log|w_2|^2\big| <
\arccos(e^{-v}),\ w_2\neq0 \Big\}\, .
\end{equation}
Then, $\cU$ is pseudoconvex, $\Phi: \cW\to \cU$ is a biholomorphism 
with $\Phi^{-1}(w_1,w_2)=(2 e^{iw_1}, w_2)$, $(w_1,w_2)\in\cU$ and
$A^2(\cU)$ is infinite dimensional.
\end{prop}

\proof
It is easily checked that $\Phi$ is holomorphic and injective. Moreover, 
we observe that
\begin{align*} 
\cW
&=\big\{(z_1, z_2) : \Re\big({z_1} e^{-i\log|z_2|^2}\big)> |z_1|^2/2,\, z_2\neq0\big\}\\
&= \big\{(r e^{i\theta}, z_2) :\,  r<2, \big|\theta-\log|z_2|^2\big| <
\arccos(r/2),\, z_2\neq0\big\}\, .
\end{align*}
The conclusion $\Phi(\cW) = \cU$ now follows
easily.  Hence, $\cU$ is pseudoconvex.
Additionally, $\Phi(2 e^{iw_1}, w_2)=(w_1,w_2)$ by direct
computation. 

Finally, setting
$Tf(w_1,w_2) = 2ie^{iw_1}f(2e^{iw_1}, w_2)$, then we obtain an isometric
isomorphism 
$$
T: A^2(\cW) \to A^2(\cU)\, ,
$$ 
so that $A^2(\cU)$ is nontrivial by Proposition \ref{nontrivial}.
\qed
\medskip

It is interesting to compare $\cW$ with the domain
\begin{equation}\label{D-infty}
D_\infty= \big\{(z_1,z_2)\in \bbC^2: \Re \big(z_1 e^{-\log
  |z_2|^2}\big)>0,  z_2\neq0\big\}\, ,
\end{equation}
and $\cU$ with the domain 
$$
D'_\infty = \big\{ (z_1,z_2)\in \bbC^2: |\Im z_1 -\log
|z_2|^2|<\frac\pi2,  z_2\neq0 \big\}\, ,
$$
biholomorphic to $D_\infty$ via
the mapping $D'_\infty\ni (z_1,z_2)\mapsto (e^{z_1},z_2)\in
D_\infty$.  
We can think of $D_\infty$ and $D'_\infty$ as the limits as 
$\mu\to+\infty$ of the domains
$$
D_\mu = \big\{ (z_1,z_2)\in \bbC^2: \Re \big(z_1 e^{-\log
  |z_2|^2}\big)>0, 
|\log|z_2|^2|\le \mu\big\}
$$
and 
$$
D'_\mu = \big\{ (z_1,z_2)\in \bbC^2: |\Im z_1 -\log
|z_2|^2|<\frac\pi2, |\log|z_2|^2|\le \mu\big\}\, ,
$$
studied in 
\cite{Ki, Barrett-Acta,Kr2,KrPe3, KrPe}.

\begin{prop}\label{Ber-trivial}
The spaces $A^2(D_\infty)$ and $A^2(D'_\infty)$ are trivial.
\end{prop}

We postpone the proof to the end of the next section.
\bigskip\\

\section{Reduction to one variable}\label{unwound}

If $\Omega$ denotes either $\cW$ or $\cU$, the Bergman
space $A^2(\Omega)$ decomposes as $\bigoplus_{j \in \bbZ}\cH^j(\Omega)$
where
\begin{align*}
\cH^j (\Omega) 
&= \big\{F\in A^2(\Omega) : \ F(w_1,e^{i\theta}w_2) = e^{ij\theta} F(w_1,w_2)\, ,\ \text{for\ }
\theta\in\bbR\big\}\\ 
&= \big\{F\in A^2(\Omega) : F(w_1,w_2)w_2^{-j} \mathrm{\ is\ locally\
  constant\ in\ }w_2 \big\} \, .
\end{align*}
Proposition \ref{nontrivial} shows that, for every $j\in\bbZ$,
$\cH^j(\cW)$ is nontrivial.  Furthermore, $T(\cH^j(\cW))
=\cH^j (\cU)$ and the restriction $T: \cH^j(\cW) \to \cH^j (\cU)$ is an
isometric isomorphism. 

We recall that the projection $Q_j: A^2(\Omega)\to \cH^j(\Omega)$ is
given by
$$
Q_j F (z_1,z_2) =
\frac{1}{2\pi} \int_0^{2\pi} F(z_1,e^{i\theta}z_2) e^{-ij\theta}\,
d\theta \, .
$$
For more details, see \cite{Barrett-Acta}.
\vspace*{.12in}

Let $\pi_1:\cU\to\bbC$ be the projection map onto the first
variable. Then
$\pi_1(\cU)$ equals the upper half-plane $\bU=\{w_1=u+iv: v>0\}$.

The fiber over each point $w_1\in\bU$ is connected (contrary to the
case of $\cW$). 
Indeed, the fiber over
$w_1=u+iv$, $v>0$, is the annulus   
\begin{align*}
\pi_1^{-1}(u+iv)
&= \big\{w_2\in\bbC :\,  \big|u-\log|w_2|^2\big| < \arccos(e^{-v})\big\}\\
&= 
\Big\{w_2\in\bbC :\,  e^{[u-\arccos(e^{-v})]/2} < |w_2|< e^{[u+\arccos(e^{-v})]/2}\Big\}. 
\end{align*}

Hence $F \in
\cH^j(\cU)$ if and only if ($F$ is square integrable and)
$F(w_1,w_2)= f(w_1)w_2^j$ for some holomorphic function
$f:\bU\to \bbC$.  In the next lemma, and in the rest of the paper, 
we denote by $A^2(\Omega,\alpha)$ the weighted Bergman space 
on the domain $\Omega$ with respect to the continuous,
positive weight $\alpha$.

\begin{lem} 
For $F\in \cH^j (\cU)$ set $L_jF(w_1,w_2) = F(w_1,w_2) w_2^{-j}$. 
Then $L_j$ is an isometric isomorphism from $\cH^j(\cU)$ to the 
weighted Bergman space $A^2(\bU,\omega_j)$, where the weight 
$\omega_j$ defined as
\begin{equation}
\omega_{-1}(u+iv)= 2\pi \arccos(e^{-v})
\end{equation}
for $j=-1$ and as
\begin{equation}
\omega_j(u+iv)= \frac{2\pi}{j+1}  e^{(j+1)u} \sinh\big[(j+1) \arccos(e^{-v})\big]
\end{equation}
for all other $j \in \bbZ$.
\end{lem}

\begin{proof}
Let $F,G\in \cH^j$, and let 
$f,g$ be holomorphic on $U$ such that $F(w_1,w_2)=
f(w_1)w_2^j$ and $G(w_1,w_2)= g(w_1)w_2^j$, $w_1\in\bU$.  We have
\begin{align*}
\la F,G\ra 
&= \int_{\bU} f(w_1)\overline{g(w_1)} \int_{\pi_1^{-1}(w_1)}
|w_2|^{2j}\, 
 dV(w_2) dV(w_1)\\ 
&= \int_{\bU} f(w_1)\overline{g(w_1)} \omega_j(w_1)\, dV(w_1)\, ,
\end{align*}
where
$$
\omega_j(u+iv) 
= 2\pi \int_{e^{[u-\arccos(e^{-v})]/2}}^{e^{[u+\arccos(e^{-v})]/2}}
\rho^{2j+1}\, d\rho.
$$
The conclusion now follows.
\end{proof}
\vspace*{.12in}

Taking into account that 
$e^{(j+1)u} = 
\big|e^{\frac{j+1}{2}w_1}\big|^2$ for all $w_1 = u+iv\in\bU$, if we
set
\begin{equation}\label{M-j}
M_jf(\z)= f(\z) e^{\frac{j+1}{2}\z} \, ,
\end{equation}
we obtain an isometric isomorphism $M_j: A^2(\bU,\omega_j)\to
A^2(\bU,\alpha_j)$.    Here
\begin{equation}\label{alpha-j}
\alpha_j(u+iv)= \frac{2\pi}{j+1} \sinh\big[(j+1) \arccos(e^{-v})\big]
\end{equation}
if $j\neq-1$, and $\alpha_{-1}(u+iv)= 2\pi \arccos(e^{-v})$.

Hence we have the following.

\begin{cor} \sl
The mapping $M_j f(\zeta) = f(\zeta) e^{[(j+1)\zeta]/2}$ defines an
isometric isomorphism $M_j : A^2(\bU, \omega_j) \to A^2(\bU,
\alpha_j)$.
\end{cor}
\vspace*{.12in}

Notice that $\alpha_j(u+iv)$ is independent of $u$ and that, with an
abuse of notation, we may
write $\alpha_j(u+iv)=\alpha_j(v)$, $v>0$.
Moreover, 
$$
0< \alpha_j(v) <\frac{2\pi}{j+1} \sinh\big[(j+1)\pi/2\big]
$$
for all $v>0$. 
This implies that $A^2(\bU, \alpha_j)$
contains the unweighted Bergman space $A^2(\bU)$. However,
$\alpha_j(v)$ is asymptotic to $\sqrt{v}$ as $v\to0^+$, so the reverse
inclusion does not hold.

We also point out that the mapping $j\mapsto \alpha_{j}$ 
is even in $j+1$, that is,
$ \alpha_j=\alpha_{-2-j}$ for all $j \in \bbZ$. 
\medskip

We conclude this section with a proof of Proposition \ref{Ber-trivial}.
\proof[Proof of Prop. \ref{Ber-trivial}]
By holomorphic invariance, it suffices to show that
$A^2(D'_\infty)=\{0\}$.
Arguing as we did for $A^2(\cU)$, we obtain that 
$A^2(D'_\infty)=\bigoplus_{j\in\bbZ} \cH^j(D'_\infty)$, where 
$$
\cH^j(D'_\infty) =\big\{ F\in A^2(D'_\infty):\, F(z_1,z_2) =f(z_1)z_2^j,\ f\mathrm{\ entire} \big\}\, .
$$
For $F\in A^2(D'_\infty)$ with $F(z_1,z_2) =f(z_1)z_2^j$, we have
\begin{align*}
\|F\|_{ A^2(D'_\infty)}^2
& =  2\pi \int_0^{+\infty} \int_{|\Im z_1 -\log r^2|<\frac\pi2}
|f(z_1)|^2\, dV(z_1)\, r^{2j+1}\, dr \\
%% & =  \pi \int_0^{+\infty} \int_{|\Im z_1 -s|<\frac\pi2} |f(z_1)|^2\, dV(z_1)\, e^{(j+1)s}\, ds \\
& = \pi \int_\bbC |f(z_1)|^2\, \int_{|\Im z_1-s|<\pi/2} e^{(j+1)s}\, ds
dV(z_1)\\ 
%% & = 2\pi \sinh \big[\pi(j+1)/2\big]  \int_\bbC |f(z_1)|^2 e^{(j+1)\Im z_1}\, dV(z_1) \\
& = 2\pi \frac{\sinh \big[(j+1)\pi/2\big]}{j+1} 
\int_\bbC \big|e^{-\frac{i}{2}(j+1)z_1} f(z_1)\big|^2 \, dV(z_1)\, ,
\end{align*}
if $j\neq-1$, and with the obvious modification otherwise.
Thus, $F \in A^2(D'_\infty)$ forces the entire
function $e^{-\frac{i}{2}(j+1)z_1} f(z_1)$ to be identically zero; hence
the conclusion.
\qed
\medskip

\vskip24pt
\section{The Bergman kernel of $A^2(\bU,\alpha_j)$}
\vskip12pt 

We now study the kernel of $A^2(\bU, \alpha_j)$. In order to do so, we
adapt the technique of \cite{Barrett-Acta}. For each $f \in A^2(\bU,
\alpha_j)$, owing to the fact that $\alpha_j$ is bounded and that it
depends only on $v$, and since $f(\cdot+iv)\in L^2(\bbR)$ for every $v$
fixed, we can consider the partial Fourier transform and set
\begin{align*}
\widehat f(\xi,v) = \int_\bbR f(u+iv) e^{-iu\xi}\, du\, .
\end{align*}

For our current purposes, we need the following simple version of the
Paley--Wiener theorem for 
weighted Bergman spaces. The equality
\begin{align*}
\widehat{\alpha}_j(-2i\xi) = \int_0^{+\infty} e^{-2v\xi}\alpha_j(v)\, dv
\end{align*}
is clearly well defined for any $\xi>0$, and it
is the Fourier transform of 
$\alpha_j$, defined to be zero on the negative reals,
extended to the lower half-plane and computed at $-2i\xi$.

\begin{prop} \sl
(1) Let $f\in A^2(\bU, \alpha_j)$. Then, for all $v>0$,
$\operatorname{supp}\widehat f(\cdot,v)\subseteq (0,+\infty)$,
$\widehat f(\cdot,v)\in
L^2\big((0,+\infty),\widehat{\alpha}_j(-2i\xi)d\xi\big)$, and
there exists $g\in
L^2\big((0,+\infty),\widehat{\alpha}_j(-2i\xi)d\xi\big)$ such that
\begin{equation}\label{PW-conv}
\widehat f(\cdot,v) \to g \quad\text{in\ }
L^2\big((0,+\infty),\widehat{\alpha}_j(-2i\xi)d\xi\big) 
\end{equation}
as $v\to0^+$.   Moreover,
\begin{equation}\label{PW-iden}
f(w) =\frac{1}{2\pi} \int_0^{+\infty} e^{iw\xi} g(\xi)\, d\xi\, . 
\end{equation}
and
\begin{equation}\label{PW-isom}
\|f\|_{A^2(\bU,\alpha_j)} =\frac{1}{2\pi}
\|g\|_{L^2((0,+\infty),\widehat{\alpha}_j(-2i\xi)d\xi)}\, . 
\end{equation}

\noindent (2) Conversely, if $g \in
L^2\big((0,+\infty),\widehat{\alpha}_j(-2i\xi)d\xi\big)$ then
\eqref{PW-iden} defines a function $f\in A^2(\bU,\alpha_j)$ such
that
\eqref{PW-isom} holds.
\end{prop}

\proof
For simplicity we write $\alpha_j=\alpha$.  
Let $f\in A^2(\bU,\alpha)$.  For every 
$\eps>0$ the function $\bU\ni\zeta\mapsto f(\zeta+i\eps)$ is in the
Hardy space 
$H^2(\bU)$.  
By the Paley--Wiener theorem, there exists 
a function $g_\eps \in L^2(0,+\infty)$ such that
\begin{equation}\label{PW-id-H2U}
f(\zeta+i\eps)= \frac{1}{2\pi} \int_0^{+\infty} e^{i\zeta\xi} g_\eps(\xi)\,
d\xi\,.
\end{equation}
Moreover, the Fourier transform $\cF( f(\cdot+i\eps))$ is supported in 
$(0,+\infty)$ and it coincides with $g_\eps$.
Now
\begin{align*}
f(u+i\eps'+i\eps)
& = \frac{1}{2\pi} \int_0^{+\infty} e^{iu\xi}e^{-\eps'\xi}
g_\eps(\xi)\, d\xi \\
& = \frac{1}{2\pi} \int_0^{+\infty} e^{iu\xi}e^{-\eps\xi}
g_{\eps'}(\xi)\, d\xi\, ,
\end{align*}
so that $e^{\eps\xi} g_{\eps}(\xi) =e^{\eps'\xi} g_{\eps'}(\xi)$ for every
$\eps,\eps'>0$.   We are thus able to set
$g(\xi) =e^{\eps\xi} g_{\eps}(\xi)$ without ambiguity. 
For every
$u+iv\in\bU$, observing that the integrals below converge 
absolutely, we have 
\begin{align*}
\cF^{-1} (g_v)(u)
& = \frac{1}{2\pi} \int_0^{+\infty} e^{iu\xi}e^{-v\xi} g(\xi)\, d\xi
= \frac{1}{2\pi} \int_0^{+\infty} e^{i(u+iv)\xi}g(\xi)\, d\xi 
\notag \\
& = \frac{1}{2\pi} \int_0^{+\infty} e^{i(u+iv-i\eps)\xi}
g_\eps(\xi)\, d\xi =f(u+iv-i\eps+i\eps)\notag \\
&  = f(u+iv)\, \label{Fou-inv}
\end{align*}
by \eqref{PW-id-H2U}. This proves both \eqref{PW-iden} and 
the equality $\widehat f(\cdot,v) = g_v$, from which 
\eqref{PW-conv} immediately follows.
Moreover, 
by Plancherel's theorem,
\begin{align*}
\|f\|_{A^2(\bU,\alpha)}^2
& = \frac{1}{2\pi}
\int_0^{+\infty} \int_0^{+\infty} \big| e^{-v\xi} g(\xi)\big|^2\, d\xi\,
\alpha(v)\, dv \\
& = \int_0^{+\infty} |g(\xi)|^2 \int_0^{+\infty} e^{-2v\xi}
\alpha(v)\, dv\, d\xi \\
& = \int_0^{+\infty} |g(\xi)|^2 \widehat\alpha(-2i\xi)\, d\xi\, . 
\end{align*}
 This proves
\eqref{PW-isom}.
The proof of part (2) follows the same lines.
\medskip
\qed
\medskip \\

Notice that in particular we have that, for $w\in\bU$,
$$
f(w) =\frac{1}{2\pi} \int_0^{+\infty} \widehat f(\xi,0)\, e^{iw\xi}\,
d\xi\, .
$$
The previous lemma allows us to prove the following result, where 
$B$ and $\Gamma$ denote the classical beta function and 
gamma function.

\begin{prop}  \sl
The kernel $K_j$ of $A^2(\bU, \alpha_j)$ can be computed as
\begin{equation}\label{jkernel}
K_j(z,w) = \frac{1}{2\pi} \int_0^{+\infty} \frac{e^{i(z-\overline{w})\xi}}{\widehat{\alpha}_j(-2i\xi)} d\xi,
\end{equation}
for $z,w \in \bU$, where for $\xi>0$ we have
\begin{align}
\frac 1{\widehat{\alpha}_j(-2i\xi)}\, &=
\frac{2^{2\xi+1} \xi(2\xi+1)}{\pi^2}
{B\Big(\xi+1+ i\frac{j+1}2,\xi+1- i\frac{j+1}2 \Big)}\label{beta-fnct}\\
&=\frac{1}{\pi^2}\frac{2^{2\xi}}{\Gamma(2\xi)} \,{\Big|\Gamma\Big(\xi
      +1+ i\frac{j+1}2\Big)\Big|^2} \, .\label{beta-fnct-2}
\end{align}
\end{prop}

\proof
Fix $v_0>0$ and let $K_j^w(z)=K_j(z,w)$. Then, for  $f\in
A^2(\bU,\alpha_j)$ and  $w \in \bU$, we have
\begin{align*}
f(w) &= \la f,K_j^w\ra_{\alpha_j} 
= \int_0^{+\infty} \int_\bbR f(x+iy) \overline{K_j^w(x+iy)}\,  dx\, \alpha_j(y)\, dy\\
&= \frac{1}{2\pi} \int_0^{+\infty} 
\int_\bbR \widehat f(x,\xi) \overline{\widehat{K}_j^w(\xi,y)}\, d\xi \,\alpha_j(y) dy\\
&= \frac{1}{2\pi} \int_0^{+\infty} 
\int_\bbR e^{-2y\xi} \widehat f(\xi,0)
  \overline{\widehat{K}_j^w(\xi,0)} d\xi  
\, \alpha_j(y)\, dy\\
&=  \frac{1}{2\pi} \int_\bbR \widehat f(\xi,0) 
\overline{\widehat{K}_j^w(\xi,0)} 
\int_0^{+\infty} e^{-2y\xi} \alpha_j(y)\, dy \, d\xi\\
& =\frac{1}{2\pi} \int_\bbR \widehat f(\xi,0) 
\overline{\widehat{K}_j^w(\xi,0)} \widehat{\alpha}_j(-2i\xi)\, d\xi
\, .
\end{align*}
Coupling this with \eqref{PW-iden}, we conclude that, on the
support of $\widehat{K}_j^w(\cdot,0)$,  
\begin{align*}
e^{iw\xi} & = \overline{\widehat{K}_j^w(\xi,0)} \widehat{\alpha}_j(-2i\xi)
= \overline{\widehat{K}_j^w(\xi,y)} e^{y\xi}
\widehat{\alpha}_j(-2i\xi)
\end{align*}
for all $y\geq0$. Therefore
\begin{align*}
\widehat{K}_j^w(\xi,y) = \frac{ e^{i(iy-\overline{w})\xi}}{
  \widehat{\alpha}_j(-2i\xi)}\, ,
\end{align*}
which gives
%we can compute $K_j^w$ by applying formula \eqref{IF-id} at $z$:
\begin{align*}
K_j^w(z)  %&= \frac{1}{2\pi} \int_0^{+\infty} \widehat{K}_j^w(x_0,\xi) e^{(z-x_0)\xi}d\xi\\
&=  \frac{1}{2\pi} \int_0^{+\infty}  \frac{ e^{i(z -\overline{w})\xi}}{ \widehat{\alpha}_j(-2i\xi)} d\xi
\end{align*}
provided the integral converges absolutely.
Let us compute $\widehat{\alpha}_j(-2i\xi)$.  We have
\begin{align*}
\widehat{\alpha}_j(-2i\xi) %=& \int_0^{+\infty} e^{-2y\xi}\alpha_j(y)\, dy\\
=& \frac{2\pi}{j+1} \int_0^{+\infty} e^{-2y\xi} \sinh\big[(j+1)
\arccos(e^{-y})\big]\, dy\\
=& \frac{2\pi}{j+1} \int_0^1t^{2\xi} \sinh\big[(j+1) \arccos(t)\big]
\, \frac{dt}{t}\\
=& \frac{2\pi}{j+1} \int_0^{\pi/2} (\cos s)^{2\xi-1}
\sinh\big[(j+1)s\big] \sin s\, ds\\
=& \frac{\pi}{\xi} \int_0^{\pi/2} {(\cos s)^{2\xi}}
\cosh\big[(j+1)s\big]\, 
ds\, .
\end{align*}

Since $\cosh\big[(j+1)s\big] = \cos(\theta s)$ with $\theta:=i(j+1)$ 
and since $\tau:=2\xi>0$, we may use formula 3.631(9)  in \cite{tables} to 
obtain
\begin{align*}
\frac{2\pi}{\tau} \int_{0}^{\pi/2} (\cos s)^{\tau} \cos(\theta s)\, ds
& = \frac{\pi^2}{2^{\tau} \tau(\tau+1)}
\frac{1}{B\left( \frac{\tau+2+ \theta}2,\,
    \frac{\tau+2+\overline{\theta}}2 \right)}\, \\
&= \frac{\pi^2}{2^{\tau}} 
\frac{\Gamma(\tau)}{\Gamma\left(\frac{\tau+2+\theta}2\right) 
\Gamma\left(\frac{\tau+2+\overline{\theta}}2\right)}\, .
\end{align*}
Formulas \eqref{beta-fnct} and \eqref{beta-fnct-2} now follow.

We are now in a position to prove the absolute convergence of the 
integral in \eqref{jkernel} by means of estimates for the weight function
$\big[\widehat{\alpha}_j(-2i\xi) \big]^{-1}$.  We set
\begin{equation*}
\eta = \frac{j+1}2 \quad\text{and}\quad \beta_\eta(\xi) =
\frac{1}{2\pi \widehat{\alpha}_j(-2i\xi)} = c \,
\frac{\xi 2^{2\xi}\big|\Gamma\big(\xi +1+ i\eta\big)\big|^2} {\Gamma(2\xi+1)}
\, .
\end{equation*}
According to Stirling's formula,
\begin{align}
\big|\Gamma(\xi +1+ i\eta)\big|^2
&= \big| \sqrt{2\pi}\exp \big \{(\xi+1/2+i\eta)\log
      (\xi+1+i\eta)-(\xi+1+i\eta)\big\} \big|^2  \Big[1+
O\Big(\frac1{\xi+1+i\eta}\Big)\Big] \notag\\ 
&\le c\, \exp\big\{2(\xi+1/2)\log |\xi+1+i\eta| - 2 \eta \arg
    (\xi+1+i\eta) -2(\xi +1) \big\}\, , \label{stir-1}
\end{align}
 for some constant $c$, independent of $\xi$ and $\eta$.  Also
\begin{align}
\frac{2^{2\xi}}{\Gamma(2\xi+1)}
& \leq c\, \exp \big\{(2\log2)\xi - (2\xi+1/2)\log
(2\xi+1)+(2\xi+1)\big\}\notag\\
& = c \exp \big\{ -(2\xi+1/2)\log (\xi+1/2)+(2\xi+1)\big\} 
\, . \label{stir-2}
\end{align}
Putting together \eqref{stir-1} and \eqref{stir-2} we
obtain that 
\begin{align*}
|\beta_\eta(\xi)|& \le c\, \xi  \exp \Big\{
2(\xi+1/2)\log \Big(\frac{|\xi+1+i\eta|}{\xi+1/2}\Big) - 2 \eta \arg
    (\xi+1+i\eta) 
 +1/2\log (2\xi+1) \Big\} \\
& \le c\, \xi^{3/2}  \exp \Big\{
2(\xi+1/2)\log \Big(\frac{|\xi+1+i\eta|}{\xi+1/2}\Big) - 2 \eta \arg
    (\xi+1+i\eta) \Big\} \\
 & \le c\, \xi^{3/2}  \exp \Big\{
2(\xi+1/2)\log \Big(1+\frac{|\eta|+1/2}{\xi+1/2}\Big) - 2 \eta \arg
    (\xi+1+i\eta) \Big\}\, .
\end{align*}
Observing that $\eta \arg
    (\xi+1+i\eta)>0$ for $\xi>0$ and that $\Re(i(z-\overline w))<0$,
    the absolute convergence of the integral in \eqref{jkernel} follows.
    Moreover, for any fixed $\eps>0$, the absolute convergence of the
    integral is uniform 
    for $\Re(i(z-\overline w))\leq -\eps$.
\qed
\medskip \\

We now show that for fixed $(z,w)$ all the values $K_j(z,w)$ 
can be obtained by evaluating a single function at the integer points. 
This further representation allows us to describe the behavior 
of $K_j(z,w)$ as $\Re(i(z-\overline w))\to 0^-$. 

Recall that we denote 
by $\bH$ the right half-plane in $\bbC$.

\begin{prop}\label{3.3}  \sl
The kernel $K_j$ of $A^2(\bU, \alpha_j)$ is given by $K_j(z,w) = \widehat
\phi_\lambda(j+1)$, where $\lambda:=-i(z-\overline{w}) \in \bH$ and 
\begin{equation}\label{phi-hat}
\phi_\lambda(s) = \frac{1}{2\pi^3} \frac{1}{\cosh^2 s}\Big[
\big(2\log(\cosh s)+\lambda\big)^{-2} + 4 \big( 2\log(\cosh
s)+\lambda\big)^{-3}
\Big] \, .
\end{equation}
The mapping $\lambda\mapsto \phi_\lambda$ is holomorphic in
$\bH$ and it takes its values in the Schwartz space 
$\mathcal{S}(\bbR)$. The same is true for the Fourier transform 
$\widehat \phi_\lambda(\xi) = \int_{\bbR} {e^{-i\xi
    s}}\phi_\lambda(s)ds$.  

Moreover, for every $j\in\bbZ$, 
$$
K_j:\, \bU\times\bU\to \bbC
$$
extends
holomorphically in $z$ and anti-holomorphically in $w$ to 
$\overline{\bU}\times\overline{\bU}\setminus \Delta$,
where $\Delta$ denotes the boundary diagonal and  the ``bar''
the topological closure.
\end{prop}

\proof
From \eqref{jkernel} and \eqref{beta-fnct}, having set
$\lambda=-i(z-\overline{w})$, we have that 
\begin{align*}
K_j(z,w)
&= \frac1 {\pi^3}\int_0^{+\infty} {2^{2\xi}
e^{-\lambda\xi}}\xi(2\xi+1){B\Big(\xi+1+ i(j+1)/2,\xi+1- i (j+1)/2
  \Big)} d\xi\\ 
&= \frac1 {\pi^3}\int_0^{+\infty}
{2^{2\xi}e^{-\lambda\xi}}\xi(2\xi+1)\int_0^{+\infty} \frac{t^{\xi+ 
    i(j+1)/2}}{(1+t)^{2\xi+2}}\, dt\, d\xi\\ 
&= \frac1 {\pi^3}\int_0^{+\infty} t^{i(j+1)/2} \int_0^{+\infty}
\frac{2^{2\xi}e^{-\lambda\xi} t^{\xi}}{(1+t)^{2\xi+2}}\xi(2\xi+1)
d\xi\, dt\\ 
&= \frac1 {\pi^3}\int_0^{+\infty} \frac{t^{i(j+1)/2}}{(1+t)^2}
\int_0^{+\infty} \xi(2\xi+1)\exp\big\{\xi \big(\log \chi(t)
-\lambda\big)\big\}\, 
d\xi\, dt \, ,
\end{align*}
where $\chi(t) = 4t/(1+t)^2$. Therefore
\begin{align*}
K_j(z,w) &=  \frac1 {\pi^3}\int_0^{+\infty}
\frac{t^{i(j+1)/2}}{(1+t)^2} \Big[\big(\log \chi(t) -\lambda\big)^{-2} -
4\big(\log \chi(t) -\lambda\big)^{-3}\Big]\, dt\\ 
&= \frac1 {2\pi^3}\int_0^{+\infty} {t^{i (j+1)/2}}{\chi(t)}
\Big[\big(\log \chi(t) -\lambda\big)^{-2} -
4\big(\log \chi(t) -\lambda\big)^{-3}\Big]\,
\frac{dt}{2t} \, . 
\end{align*}
Setting $t=e^{2s}$ and observing that $\chi(e^{2s}) =
\left( 2e^{s}/(1+e^{2s}) \right)^2 = \cosh^{-2}s$, we have
\begin{align*}
K_j(z,w) &= \frac1 {2\pi^3}\int_{\bbR}
\frac{e^{i(j+1)s}}{\cosh^2 s}
\Big [\big(2\log\cosh s +\lambda\big)^{-2} +4
\big(2\log\cosh s+\lambda)^{-3}\Big]\, ds\\ 
&= \widehat \phi_\lambda(j+1) \, ,
\end{align*}
as claimed, taking into account that $\phi_\lambda$ is even. 

Finally, it is clear that $\phi_\lambda(s)$ is a Schwartz function
in $s$
 when $\lambda$ is bounded away  from the set
$(-\infty,0]$. It is
also easy to see that the mapping
$\lambda\mapsto \phi_\lambda\in\cS(\bbR)$
 is holomorphic in $\lambda$ in the slit plane $\bbC\setminus(-\infty,0]$.
Therefore
$K_j(z,w)$ extends holomorphically in $z$ and anti-holomorphically in
$w$ 
in a neighborhood of each
point $(z,w)$ of $\overline{\bU}\times\overline{\bU}$
except those for which $\lambda=-i(z-\overline w)=0$, that is,
$z-\overline{w}=0$.  This last implies that $z=w\in \p\bU$ so that
 $K_j(z,w)$ extends holomorphically in $z$ and anti-holomorphically in
$w$   to a neighborhood of each
point $(z,w)$ in  $\overline{\bU}\times\overline{\bU}\setminus\Delta$. 
\qed
\medskip \\

We now study the dependence of $K_j$ on the index $j$. Recall that we
have set $\lambda = -i(z-\overline w)$.
\begin{cor}\label{ptwise-estimate}  \sl
Let 
\begin{equation}\label{b}
b_\lambda= \max \big\{ \arccos \big( e^{- \Re\lambda/2} \big), \,
\min \big\{ 
\textstyle{|\Im\lambda|/2}, \, \textstyle{\pi/2} \big \} \big \}\, .
\end{equation}
Then, for $0<b<b_\lambda$ and for $(z,w)
\in \overline{\bU} \times \overline{\bU}\setminus\Delta$ we have 
\begin{align}\label{26}
\lim_{j\to\pm\infty} |K_j(z,w)|  e^{b|j+1|} =0 \, .
\end{align}
As a consequence, for $(z,w)\in
\overline{\bU}\times\overline{\bU}\setminus\Delta$, 
\begin{equation}\label{rootcriterion}
\limsup_{j\to\pm\infty} |K_j(z,w)|^{1/|j+1|} \le e^{-b_\lambda}.
\end{equation}
\end{cor}

\proof 
We set $S_b = \{s+it : |t|<b\}$, and 
$I_+=i\left(\frac \pi 2, \pi\right)$, $I_-=i\left(-\pi, -\frac \pi
  2\right)$ to denote two intervals on the imaginary axis.

The function $\log\cosh s$ extends holomorphically to
$S_\pi \setminus \left(I_+\cup I_-\right)$, since the function
$\cosh(s+it) = \cosh s\cos t+i\sinh s\sin t$ maps $S_\pi \setminus
\left(I_+\cup I_-\right)$ to $\bbC \setminus (-\infty,0]$. For each
$\lambda \in \overline{\bH}\setminus\{0\}$, the functions $s \mapsto
\phi_{\lambda}(s)$ and $s\mapsto
s \phi_{\lambda}(s)=\widetilde\phi_\lambda (s)$, extend holomorphically to
$S_{\pi/2}$.
We still denote by $\phi_\lambda$ and $\widetilde\phi_\lambda$ such
extensions. 

We claim that $\phi_\lambda$ and $\widetilde\phi_\lambda$ belong to the
Hardy space $H^2(S_b)$, for every $b<b_\lambda$.
Assuming the claim, we complete the proof.

By the classical Paley--Wiener theorem for $H^2(S_b)$, 
$e^{\pm b\xi}\widehat \phi_{\lambda}(\xi)$
and $e^{\pm b\xi} \frac{d}{d \xi}\widehat\phi_{\lambda}(\xi)$ belong to $L^2(\bbR)$.
 If we set $f_\pm(\xi) =e^{\pm b\xi}\widehat \phi_{\lambda}(\xi)$,
 then  $f_\pm \in W^1(\bbR)$. By the Sobolev embedding theorem it
 follows that $f_\pm $ is a continuous function vanishing at infinity.
Hence
$$
\lim_{\xi\to\pm\infty} e^{ b|\xi|}\widehat \phi_{\lambda}(\xi) =0\, ,
$$
which gives \eqref{26}.
\medskip 

It only remains to prove the claim.
Notice that, assuming $|t|<\pi/2$, we have that 
\begin{align*}
\big| \Re\big(2\log\cosh(s+it)+\lambda\big) \big|
& = \log\big(\sinh^2 s + \cos^2 t\big) +\Re\lambda \ge \eps_0 
\end{align*}
if $|\cos t|\ge e^{\eps_0/2} e^{-\Re\lambda/2}$, and that
\begin{align*}
\big| \Im\big(2\log\cosh(s+it)+\lambda\big) \big|
& \ge  |\Im\lambda| - 2\big|\arctan(\tanh s \tan t) \big| 
\ge  |\Im\lambda| - 2|t|  \ge \eps_0 \, ,
\end{align*}
for some $\eps_0>0$, if $|t|<  |\Im\lambda|/2$.
The claim now follows easily by Plancherel's theorem and the last
two inequalities.
\qed
\medskip \\

We conclude this section by describing the behavior of $K_j$ near
the extended boundary of $\bU\times \bU$. 
In order to do so, we first expand at infinity and then restrict to a 
special case that allows explicit computations. Recall that we denote 
by $\bH$ the right half-plane and we write $\lambda=-i(z-\overline{w})$.

\begin{lem}\label{technical-lemma}   \sl
Let $K_j$ be the Bergman kernel for $A^2(\bU,\alpha_j)$.
Let $N\ge2$ and $\eps>0$ be fixed.
Then there exist:
\begin{itemize}
\item[(i)]  Schwartz functions $\psi_1,\dots,\psi_N$;
\item[(ii)]  a  Schwartz function $
\Psi_{N,\lambda}$  holomorphic in
$\lambda\in\bH$ and converging to $\psi_N$ in $\cS(\bbR)$ as $\lambda \to
\infty$ within the half-plane 
$\overline{\bH}_\eps = \{\lambda: Re(\lambda)\ge \eps\} \subset
\bH$;
\end{itemize}
such that 
\begin{align}\label{expr-Kj}
K_{j}(z,w) &= 
\sum_{n=2}^{N-1}\frac{\psi_{n}(j+1)}{(z-\overline{w})^{n}} 
+ \frac{\Psi_{N,\lambda}(j+1)}{(z-\overline{w})^N} \, ,
\end{align}
for $z,w\in\bU$.  
Explicitly, 
\begin{align*}
\psi_n(\xi) &= \frac{(-i)^n(n-1)}{2\pi^3} \big[I_{n-2}(\xi) - 2 (n-2)
I_{n-3}(\xi)\big], 
\quad\text{where}\quad
I_m(\xi) =\int_\bbR e^{-i\xi s}\frac{\big(2 \log\cosh
  s\big)^m}{\cosh^2s}\, ds\, .
\end{align*}
\end{lem}

\proof
For $s\in\bbR$ set 
$\displaystyle{D_s =\big(2\sinh s\big)^{-1} \frac{\p}{\p s}}$.  
We use  \eqref{phi-hat} and the expansion $(1+x)^{-1}
=\sum_{n=0}^{N-1} (-x)^n+ (-x)^N(1+x)^{-1}$ to obtain that 
\begin{align}
\phi_\lambda(s) 
& =\frac{1}{\pi^3} D_s^2 \big( 2\log\cosh s+\lambda\big)^{-1}\notag \\
&=\frac1{\pi^3\lambda} D_s^2 
\Big( 1+\frac{2\log\cosh s}{\lambda}\Big)^{-1}\notag\\
&= \sum_{n=2}^N\frac{a_n(s)}{\lambda^{n}} +
\frac{A_{{N+1},\lambda}(s)}{\lambda^{N+1}}\, , \label{phi-lam}
\end{align}
where 
\begin{align*}
a_n(s) 
&= \frac{(-1)^{n-1}}{\pi^3}  D_s^2 \Big[\big(2\log\cosh s\big)^{n-1}
\Big] \\
&= \frac{(-1)^n(n-1)}{2\pi^3\cosh^2 s} \bigg[ \big(2\log\cosh s\big)^{n-2}-2
  (n-2) \big(2\log\cosh s\big)^{n-3} \bigg]\, , 
\end{align*}
and 
\begin{align*}
A_{{N+1},\lambda}(s) 
&= \frac{\lambda^N}{\pi^3} D_s^2
\Bigg[ \Big(\frac{-2\log\cosh
      s}{\lambda}\Big)^N \Big(1+\frac{2\log\cosh
      s}\lambda\Big)^{-1}\Bigg] \\
&= \frac{(-1)^N}{\pi^3} D_s^2
\Bigg[ \big(  2\log\cosh
      s\big)^N \Big(1+\frac{2\log\cosh
      s}\lambda\Big)^{-1}\Bigg] \\
&= \frac{P_{N+1}\Big(1+ [2\log\cosh s]/\lambda\Big)}{\cosh^2s
  \Big(1+ [2\log\cosh s]/\lambda\Big)^3 } \, .
\end{align*}
Here $P_{N+1}(\zeta)$ is a polynomial of degree  $2$  with coefficients
integral powers of $\log\cosh s$  such
that
$$
P_{N+1}(1)= \frac{(-1)^{N+1} N}{2\pi^3} \Big[\big(2\log\cosh s\big)^{N-1}
-2(N-1)\big(2\log\cosh s\big)^{N-2} \Big]\, .
$$
For $N\geq 1$, we have
$A_{{N+1},\lambda} \to a_{N+1}$ in $\mathcal{S}(\bbR)$ as $\lambda \to
\infty$ within the closed half-plane $\overline{\bH}_\eps$. 

Therefore, taking the Fourier transform in \eqref{phi-lam} and
recalling \eqref{phi-hat}, we obtain \eqref{expr-Kj},
where 
\begin{align*}
 \psi_{n}(\xi) &= i^n\widehat a_n(\xi) =
\frac{(-i)^n(n-1)}{2\pi^3} \big[I_{n-2}(\xi) - 2 (n-2)
I_{n-3}(\xi)\big]\, ,\\
I_m(\xi) &=\int_\bbR e^{-i\xi s}\frac{\big(2 \log\cosh
  s\big)^m}{\cosh^2s } \, ds\, .
\end{align*}
Moreover, $\Psi_{N,\lambda}= i^N \widehat A_{N,\lambda}$ are again
Schwartz functions such that,
for each $N\geq2$, $\Psi_{N,\lambda} \to \psi_{N}$ in
$\mathcal{S}(\bbR)$ as $\lambda \to \infty$ within a half-plane
$\bH_\eps$. 
\qed
\medskip \\

\begin{thm}\label{kernels} \sl
Let $K_j$ be the Bergman kernel for $A^2(\bU,\alpha_j)$. 
There exists a holomorphic function $f_j:\bH \to \bbC$ such that
\begin{align}
K_{j}(z,w) = \frac{f_j\big(-i(z-\overline{w})\big)}{(z-\overline{w})^2}
\end{align}
and
\begin{align}
\lim_{\overline{\bH}_\eps \ni \lambda \to \infty}f_j(\lambda) 
=\frac1{\pi^3} \frac{\pi\frac{j+1}2}{\sinh(\pi\frac{j+1}2)}
\end{align}
for all $\eps>0$.
Moreover, $f_j$ extends holomorphically to a neighborhood of each point of 
$\overline{\bH}\setminus\{0\}$. The product 
$\sqrt{\lambda}\, f_j(\lambda)$ is bounded near $0$ in 
$\overline{\bH}$ and 
$\lim_{\bbR^+ \ni \lambda \to 0} \sqrt{\lambda}\, f_{-1}(\lambda)<0$. 

As a consequence:
\begin{enumerate}
\item the function $(z,w) \mapsto K_{j}(z,\overline{w})$ extends 
holomorphically to a neighborhood of each point 
$(z,w) \in \partial \bU \times \partial \bU$ with $z\neq \overline{w}$;
\item the product $\big(-i(z-\overline{w})\big)^{5/2} K_{j}(z,w)$ remains 
bounded as $z-\overline{w} \to 0$ in $\overline{\bU}$ and, for $j=-1$, 
its limit as $z-\overline{w} \to 0$ in $i \bbR^+$ is a strictly 
positive real number;
\item for all $w \in \bU$, $\lim_{\bU \ni z \to \infty}K_j(z,w)=0$ 
and, for all $w \in \partial \bU$ and $\eps>0$, 
$\lim_{\bU_\eps \ni z \to \infty}K_j(z,w)=0$; similar considerations 
apply to the limits as $w \to \infty$ with $z \in \overline \bU$ fixed. 
\end{enumerate}
\end{thm}

\noindent
{\bf Remark.}  Statement (1) above was already obtained in 
Proposition \ref{3.3} and we repeated it here for the sake of
completeness.  Statement (2) shows that $K_{-1}$ is singular as $z$ and $w$
tend to the same point on the boundary of $\bU$ and that for each $j$
the (possible)
singularity of $K_j(z,w)$ is not worse that
$\big(-i(z-\overline{w})\big)^{-5/2}$.  
Finally, (3)
describes the behavior of $K_j(z,w)$ as $\bU\ni z\to\infty$.\bigskip

\proof
Owing to Lemma \ref{technical-lemma}, in order to prove the first 
statement it suffices to set $f_j(\lambda) =\Psi_{2,\lambda}(j+1)$ 
and to compute $\psi_2(\xi) = - (1/[2\pi^3]) I_0(\xi)$. We observe
that $I_0(0) = \int_\bbR  1/[\cosh^2 s] \, ds = 2$. 
For all $\xi \in \bbR$ other than $0$,
we make use of the fact that the integrand in $I_0(\xi)$
extends to
$\bbC$ except the points $\left\{i k \frac \pi 2\right\}_{k \in
  \bbZ}$.
If we integrate along the rectangle through $-R,R,R+i\pi,-R+i\pi$
and we let $R\to+\infty$ in $\bbR$ we may conclude that 
$$
I_0(\xi) =\int_\bbR \frac{e^{-i\xi s}}{\cosh^2 s} ds = \frac {2\pi i}
{1-e^{\xi\pi}} \Res_{i \pi/2}\left(\frac{e^{-i\xi
      s}}{\cosh^2 s}\right). 
$$
Taking into account that $\cosh(z+i \frac \pi 2) = i \sinh z$ and
that $1/\sinh^2 z - 1/z^2$ is holomorphic near $z=0$,
we obtain that 
$$
\Res_{i \pi/2}\left(\frac{e^{-i\xi s}}{\cosh^2 s}\right)
=-e^{\xi \pi/ 2} \Res_0\left(\frac{e^{-i\xi z}}{\sinh^2 z}\right)
= -e^{\xi \pi/ 2}\Res_0\left(\frac{e^{-i\xi z}}{z^2}\right) 
= e^{\xi\pi/ 2} i \xi\, .
$$
Therefore
$$
\psi_2(\xi) 
= -\frac{1}{2\pi^3} I_0(\xi)
= -\frac{1}{\pi^3}\frac {\pi e^{\xi\pi/ 2} \xi} {e^{\xi\pi}-1} 
= -\frac{1}{\pi^3} \frac{\xi \pi/ 2}{\sinh(\xi\pi/ 2)}
$$
for all $\xi \in \bbR$.

As for the behavior of 
$f_j(\lambda) =\Psi_{2,\lambda}(j+1) = - \widehat A_{2,\lambda}(j+1)$ 
near the finite boundary, we observe that
$$
A_{2,\lambda}(s) = \frac{1}{2\pi^3\cosh^2s}\Bigg[\Big(1+\frac{2\log\cosh
      s}\lambda\Big)^{-2} + \frac 4 \lambda \Big(1+\frac{2\log\cosh
      s}\lambda\Big)^{-3} \Bigg]
$$
admits a transform even if $\Re\lambda = 0, \Im\lambda \neq 0$. 
Moreover, we shall prove that $\sqrt{\lambda}\, \widehat A_{2,\lambda}(\xi)$ 
stays bounded as $\lambda \to 0$ and that 
$\lim_{\bbR^+ \ni \lambda \to 0} \sqrt{\lambda}\,A_{2,\lambda}(0)>0$. 
As $\lambda \to 0$, the only relevant part in 
$\sqrt{\lambda}\, \widehat A_{2,\lambda}(\xi)$ is
\begin{align*}
&\frac{2}{\pi^3\sqrt{\lambda}}\int_\bbR 
\frac{e^{-i\xi s}}{\cosh^2 s}\Big(1+\frac{2\log\cosh s}\lambda\Big)^{-3} ds \\
=&\frac{4}{\pi^3\sqrt{\lambda}}\int_0^{+\infty} \frac{\cos(\xi  s)}{\cosh^2s}
\Big(1+\frac{2\log\cosh
      s}\lambda\Big)^{-3} ds \\
=& \frac{4}{\pi^3\sqrt{\lambda}}\int_0^{1}\cos(\xi \arctanh t) 
\Big(1-{\log(1-t^2)}/\lambda\Big)^{-3} dt.
\end{align*}
 Now $ -\log(1-t^2) = \sum_{n\geq 1}\frac{t^{2n}}{n}\geq t^2$ implies that
$$
\Big\vert1-{\log(1-t^2)}/\lambda\Big\vert^2 \ge 
\Bigg(1+ t^2\frac{\Re\lambda}{|\lambda|^2}\Bigg)^2 + 
\Bigg(t^2\frac{\Im\lambda}{|\lambda|^2}\Bigg)^2 \ge 1 
+ \frac{t^4}{|\lambda|^2}
$$
for all $t \in (0,1)$. Hence, for appropriate positive constants, 
\begin{align*}
\Big\vert \sqrt{\lambda}\, \widehat A_{2,\lambda}(\xi) \Big\vert 
& \le
 \frac{C}{\sqrt{\lambda}} \int_0^{1}\Big(1 + \frac{t^4}{|\lambda|^2}\Big)^{-\frac 3 2}\,  dt\\
& \le \frac{C}{\sqrt{\lambda}} \int_0^{\sqrt{|\lambda|}} dt 
+ \frac{C}{\sqrt{\lambda}} \int_{\sqrt{|\lambda|}}^1 
\Big(1 + \frac{t^4}{|\lambda|^2}\Big)^{-\frac 3 2}
\Big(\frac{t}{\sqrt{|\lambda|}} \Big)^3\, dt \\
& \le C + \frac{C} {4} \int_{0}^1\Big(1+
\tau\Big)^{-\frac 3 2} \, d\tau \\
& \le  C\, .
\end{align*}
Moreover, for $\lambda \in \bbR^+$ sufficiently small and $t \in
(0,\sqrt{\lambda})$, the function 
$$
1-{\log(1-t^2)}/\lambda  = 1+ \frac 1{\lambda}\sum_{n\geq 1}\frac{t^{2n}}{n} 
\le 1 + \sum_{n\geq 1}\frac{\lambda^{n-1}}{n}
$$
takes values in an interval $(0,\varepsilon)$ with $\varepsilon>0$, so that
\begin{align*}
\Big| \sqrt{\lambda}\, \widehat A_{2,\lambda}(0) \Big|
& = 
\frac{4}{\pi^3\sqrt{\lambda}}\int_0^{1}\Big(1-{\log(1-t^2)}/\lambda\Big)^{-3}
dt + o(\sqrt{\lambda}) \\
& 
\ge
\frac{4}{\pi^3\sqrt{\lambda}}\int_0^{\sqrt{\lambda}}\Big(1-{\log(1-t^2)}/\lambda\Big)^{-3}
dt + o(\sqrt{\lambda}) \\
& \ge C\, ,
\end{align*}
for an appropriate positive constant $C$. 
\qed
\medskip

\vskip24pt
\section{Back to the worm domain}
\vskip12pt

We can now express the Bergman kernel of the ``unwound'' worm
$\cU$ as a series.  In this part of the paper we write 
$z=(z_1,z_2)$, $w=(w_1,w_2)$
to denote points in $\bbC^2$.  This change of notation with respect to the
previous sections should cause no confusion.   Recall that $\cU$ is
defined in \eqref{U-def}.

\begin{prop}\label{cU-kernel}	\sl
The Bergman kernel of $\cU$ is given by 
\begin{align}
K_{\cU}(z,w) &= \frac 1 {z_2 \overline w_2} \sum_{j \in \bbZ} K_j
(z_1,w_1) \Big(e^{-\frac12 (z_1+\overline w_1)} z_2 \overline w_2\Big)^{j+1}\, ,
\end{align}
for  $z = (z_1,z_2)$, $w= (w_1,w_2)$ in $\cU$, where for each
$w\in\cU$ fixed (or $z\in\cU$ fixed) the series
converges in the $L^2(\cU)$-norm, absolutely and uniformly on compact
subsets of 
$\cU$.    
\end{prop}

\proof
Considering the decomposition $A^2(\cU)=\bigoplus_{j\in \bbZ}
\cH^j (\cU)$ and  the isometry $M_jL_j: \cH^j(\cU)\to
A^2(\bU,\alpha_j)$ given by
\begin{align*}
M_jL_jF(w_1,w_2) = F(w_1,w_2) w_2^{-j} e^{[(j+1)w_1]/2}\, ,
\end{align*}
we obtain that 
 the Bergman kernel  of $\cH^j(\cU)$ is given by
$$
U_j(z,w) = K_j (z_1,w_1) e^{- [(j+1)/2] (z_1+\overline w_1)} (z_2
\overline w_2)^j\, .
$$
We are going to show that  the sum 
$\sum_{j \in \bbZ} U_j(\cdot,w)$ converges to $K_\cU(\cdot,w)$ in
$L^2(\cU)$ for any $w\in\cU$ fixed.  This will imply that the series
converges also absolutely and uniformly on compact subsets.

It is easy to see that $\sum_{|j|\le n} U_j(\cdot,w)$ weakly converges
to $K_\cU(\cdot,w)$, as $n\to +\infty$, for $w\in\cU$ fixed.  For
let $\cP_\cU$ denote the Bergman projection on $\cU$ and
let $f\in L^2(\cU)$.  Then its projection on $A^2(\cU)$ is given by
$$
\cP_\cU f(w) = 
\la f,K_\cU (\cdot,w) \ra
=  \sum_{j\in\bbZ} f_j(w)
$$
with $f_j\in \cH^j$.   Now
$$
\big\la f, \sum_{|j|\le n} U_j(\cdot,w)\big\ra
=   \sum_{|j|\le n} \la f, U_j(\cdot,w)\ra = \sum_{|j|\le n} f_j(w)
\to \cP_\cU f(w) 
$$
as $n\to+\infty$.  Hence there exists $C>0$ independent of $n$ such
that 
$$
\sum_{|j|\le n} \| U_j(\cdot,w) \|_{L^2(\cU)}^2 =
\big\| \sum_{|j|\le n} U_j(\cdot,w) \big\|_{L^2(\cU)}^2 \le C\,.
$$
Therefore $\sum_{|j|\le n}  U_j(\cdot,w) $ converges in $L^2(\cU)$,
necessarily to 
$K_\cU (\cdot,w)$. 
\qed
\medskip \\

We now study the pointwise regularity of $K_{\cU}$ at the boundary. 
In the statement, $\bU_\varepsilon = \{\zeta: \Im \zeta >\varepsilon\}$ 
with $\varepsilon>0$.
Moreover, we set
\begin{multline}\label{Sigma}
\Sigma=
\Big\{ (z,w)\in \p\cU\times\p\cU: \exists\ v\geq 0 \mathrm{\ s.t.\ } \Im z_1= \Im w_1=v,\\
\Re z_1 -\log|z_2|^2= \Re w_1 -\log|w_2|^2 = \pm \arccos (e^{-v}),\, |\log|z_2|^2-\log|w_2|^2| \leq 2\arccos(e^{-v})
\Big\}\, .
\end{multline}

\noindent
{\bf Remark.} The set $\Sigma$ contains the diagonal $\Delta$ 
of $\p\cU\times\p\cU$; but also by other 
points $(z,w)$ of $\p\cU\times\p\cU$, e.g., those such that
$z_1 = w_1\in \p \bU$ and $|z_2|=|w_2|$. See Figure 2 for other cases.
%%%%%%%%%%%%%%%%%%%%%%%%%
\begin{figure}[h]
  \begin{center}
  %\fbox{
  \includegraphics[height=6cm]{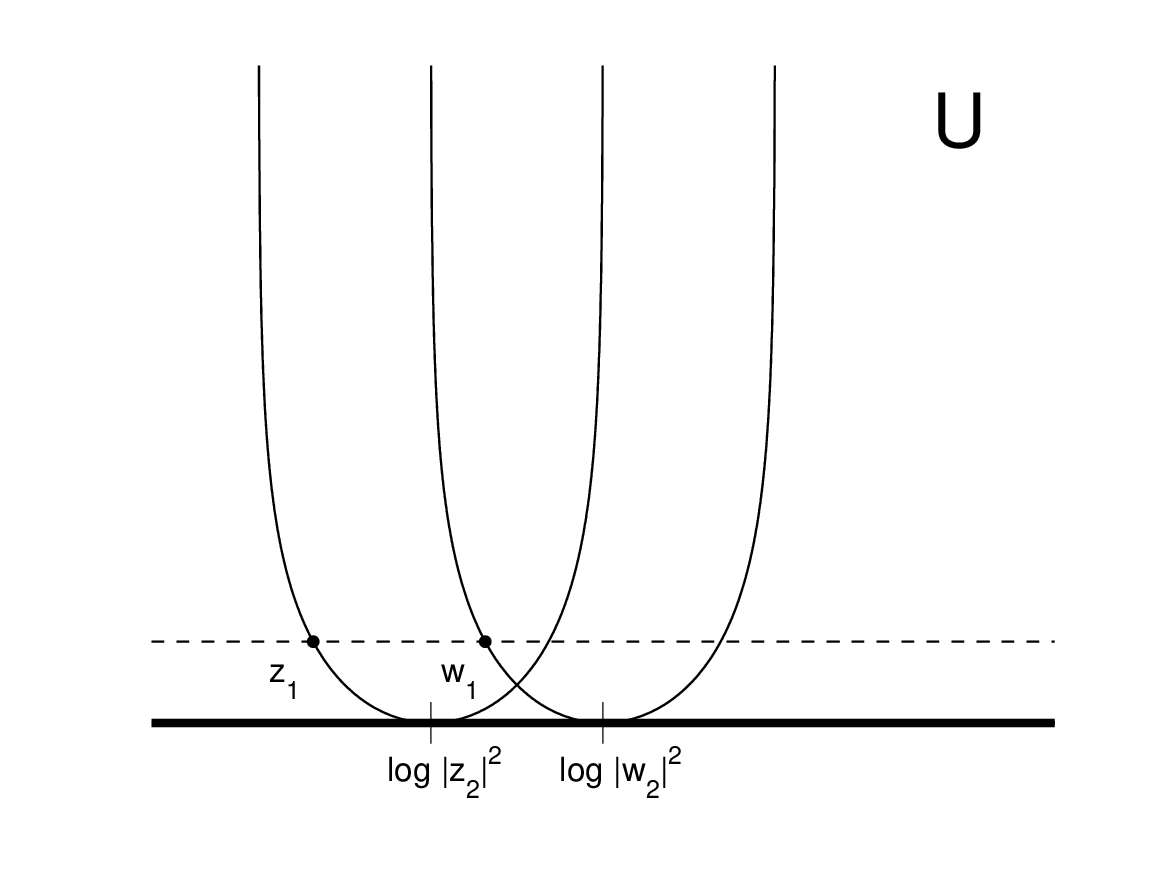}
  %}
\end{center}
  \caption{\footnotesize The set $\Sigma$ is defined to include $(z,w)$ if and only if: $z_1, w_1 \in \overline{\bU}$ lie on the same horizontal line; $z_1, w_1$ belong both to the left arcs (or both to the right arcs) of the boundaries of $\pi_1(\pi_2^{-1}(z_2)), \pi_1(\pi_2^{-1}(w_2))$; $w_1$ belongs to $\pi_1(\pi_2^{-1}(z_2))$ or $z_1 \in\pi_1(\pi_2^{-1}(w_2))$.}
\end{figure}
%%%%%%%%%%%%%%%%%%%%%%%%%
\bigskip

\begin{thm}\label{cU-kernel-boundary}	
(1) The kernel function $K_{\cU}(z, w)$ extends
holomorphically in $z$ and antiholomorphically in $w$ near each 
point $(z,w)$ in $\overline\cU \times \overline\cU\setminus \Sigma$. 

\noindent (2) There exist a holomorphic function
$G: \cU\times \cU \to \bbC$ with
\begin{align}
K_{\cU}(z,w)= \frac{G(z,w)}{z_2 \overline w_2 (z_1-\overline w_1)^2} \, ,
\end{align}
and a holomorphic function $g$ on 
$A:= \{\zeta : e^{- \pi/2}<|\zeta|<e^{\pi/ 2}\}$
such that:
\begin{itemize}
\item[$(a)$]
$G(z,w)$ stays bounded as either $z_1$ or $w_1$ tends to $\infty$;
\item[$(b)$] if $z_1 - \overline w_1 \to \infty$ within a half-plane
  $\bU_\varepsilon$ and if $e^{- \frac12 (z_1+\overline w_1)} z_2
  \overline w_2\to \zeta \in A$  then $G(z,w) \to g(\zeta)$;
\item[$(c)$] $g(\zeta) -
[e^{- \pi/ 2} \zeta]/[\pi^2(1-e^{-\pi/ 2}\zeta)^2]
- [e^{\pi/ 2} \zeta]/[\pi^2(1-e^{\pi/ 2}\zeta)^2]$ extends
holomorphically to a neighborhood of $\overline{A}$. 
\end{itemize}
As a consequence, $K_{\cU}$ 
tends to $0$ near each point $(z,w)$ or $(w,z)$ with 
$z_1=\infty, z_2\in \bbC^*\cup\{\infty\}, w \in \cU$. 
\end{thm}

\proof 
We wish to study the behavior of
$$
\sum_{j \in \bbZ} K_j(z_1,w_1) \Big(e^{-\frac12 (z_1+\overline w_1)}
  z_2 \overline w_2\Big)^{j+1} 
$$
as $z$ and $w$ in $\cU$ approach the boundary. It follows from Corollary 
\ref{ptwise-estimate} that for all $(z_1,w_1)\in
\overline{\bU}\times\overline{\bU}\setminus\Delta$
$$
\limsup_{j\to\pm\infty} |K_j(z_1,w_1)|^{1/|j+1|} \le e^{-b_\lambda}\, ,
$$
where $\lambda =-i(z_1-\overline w_1)$ and $b_\lambda$ is as in 
\eqref{b}. 
We will now complete the study of convergence, proving that
\begin{equation}\label{claim1}
e^{-b_\lambda}<\Big |e^{-\frac12 (z_1+\overline w_1)} z_2 \overline w_2 \Big| < 
e^{b_\lambda}\, ,
\end{equation}
for all $(z,w) \in \overline\cU \times \overline\cU\setminus\Sigma$.
For $(z,w)\in\cU\times\cU$ we have that
\begin{align*}
 e^{- \frac12 (z_1+\overline w_1)} z_2 \overline w_2 
&= 
e^{- \frac12  (z_1 +\overline w_1)}
e^{\frac12 (\log |z_2|^2+\log |w_2|^2)}
\frac{z_2\overline w_2}{|z_2\overline w_2|} \\ 
&= \exp\Big\{ {\textstyle\frac12}\big(\log|z_2|^2-\Re z_1 
+ \log|w_2|^2-\Re w_1 \big)
-{\textstyle \frac{i}{2}}\big( \Im z_1- \Im w_1\big)\Big\} 
\frac{z_2\overline w_2}{|z_2\overline w_2|}
\, , 
\end{align*}
where $\big|\log|z_2|^2 -\Re z_1\big| < \arccos(e^{-\Im z_1})$
and $\big|\log|w_2|^2-\Re w_1\big| < \arccos(e^{-\Im w_1})$.
Hence, using the concavity of the function $r \mapsto \arccos(e^r)$ we
obtain 
\begin{align}
\Big |e^{- \frac12 (z_1+\overline w_1)} z_2 \overline w_2 \Big|
& <
\exp\big\{ {\textstyle\frac12}
\big( \arccos(e^{-\Im z_1}) +\arccos(e^{-\Im w_1}) \big) \big\} \notag\\
& \le \exp\big\{\arccos\big (e^{-\frac12(\Im z_1+\Im w_1)}\big)
\big\}\notag \\
&  = \exp\big\{\arccos\big( e^{- \frac12\Re \lambda}\big)
\big\}\notag \\
& \leq e^{b_\lambda}
\, ; 
\end{align}
and similarly 
$\Big |e^{- \frac12 (z_1+\overline w_1)} z_2 \overline w_2 \Big|
> \exp\big\{-\arccos\big (e^{- \frac12\Re \lambda}\big)\big\}\geq 
e^{- b_\lambda}$ for all $(z,w)\in \cU \times \cU$.

The first inequality in the display above remains 
strict as either $z$ or $w$ tends to $\p\cU$ and if either $z_1$ or
$w_1$ tends to infinity.

Now let us consider $z$ and $w$ in $\p\cU$.  The equality 
$$
\Big |e^{- \frac12 (z_1+\overline w_1)} z_2 \overline w_2 \Big| = 
\exp\big\{\pm\arccos\big (e^{- \frac12\Re \lambda}\big)\big\}
$$ 
holds if and only if there exists $v\geq0$ such that
\begin{equation}\label{specialcase}
\Im z_1 = \Im w_1 = v \quad \text{and}\quad 
\log|z_2|^2 -\Re z_1 = \log|w_2|^2-\Re w_1 =
\pm \arccos(e^{-v})\, .
\end{equation}
According to formula \eqref{b}, 
$\arccos\big (e^{- [1/2]\Re \lambda}\big)=b_\lambda$
if and only if $\Im |\lambda|/2 \leq \arccos\big (e^{- [1/2]\Re \lambda}\big)$, 
which is equivalent in the special case \eqref{specialcase} to 
$\big|\log| z_2|^2 - \log |w_2|^2\big| \leq 2\arccos\big (e^{-v}\big)$. 
This proves \eqref{claim1} and also part (1) of the statement.

In order to prove (2) we further study the points at infinity by means of the expansion
\begin{align*}
K_{\cU}(z,w)
&=\sum_{j \in \bbZ} \frac {f_j(-i(z_1-\overline w_1))}{(z_1-\overline
  w_1)^2z_2 \overline w_2} \big(e^{- 1/2 (z_1+\overline w_1)}
  z_2 \overline w_2\big)^{j+1} \, , 
\end{align*}
where $f_j(\lambda) \to [k\pi/2]/[\pi^3 \sinh(k\pi/ 2)]$ as $\lambda
\to \infty$ within a half-plane 
$\bH_\eps$. If we set $G(z,w) = z_2 \overline w_2
(z_1-\overline w_1)^2 K_{\cU}(z,w)$, then 
\begin{align*}
&\lim_{e^{- \frac12 (z_1+\overline w_1)} z_2 \overline w_2\to \zeta}
G(z,w) = \sum_{j \in \bbZ} f_j(-i(z_1-\overline w_1)) \zeta^{j+1}  
\end{align*}
for 
$$
\exp\big\{
-\arccos\big(e^{-(\Im z_1+\Im w_1)/2}\big) \big\} <|\zeta|<\exp\big\{
\arccos\big(e^{-(\Im z_1+\Im w_1)/2}\big) \big\}\,.
$$

Moreover, $\sum_{j \in \bbZ} f_j(\lambda) \zeta^{j+1}$ tends
to $g(\zeta)=\frac{1}{\pi^3}\sum_{k \in \bbZ} \frac{k
  \pi/2}{\sinh(k\frac\pi 2)} \zeta^k$ as $\lambda\to \infty$   
within a half-plane $\bH_\eps$. 
We have that
\begin{align*}
\sum_{k>0} \frac{k\pi/ 2}{\sinh(k\pi/ 2)} \zeta^k 
&= 
\pi \zeta\frac{\p}{\p \zeta}\sum_{k>0} \frac 1
{e^{k\pi/ 2} - e^{-k\pi/ 2}}  \zeta^k = \pi
\zeta\frac{\p}{\p \zeta}\sum_{k>0} \frac 1 {1 - e^{-k\pi}}
(e^{-\pi/ 2}\zeta)^k\\ 
&= \pi \zeta\frac{\p}{\p \zeta}\sum_{k>0, m \geq 0}
e^{-km\pi}  (e^{- \pi/ 2}\zeta)^k 
= \pi \zeta\frac{\p}{\p \zeta}\sum_{m \geq 0} \frac {1}
{1-e^{-(m + 1/2)\pi}\zeta}\\ 
&= \pi \zeta \sum_{m \geq 0}
\frac{e^{-(m+ 1/2)\pi}}{(1-e^{-(m + 1/2)\pi}\zeta)^2} =
\frac{\pi e^{- \pi/ 2} \zeta}{(1-e^{- \pi/ 2}\zeta)^2} +
f(\zeta) \, , 
\end{align*}
where all the series converge absolutely and uniformly on compact sets
in the annulus $A$ and $f$ is holomorphic in a neighborhood of
$\overline{A}$. Thus 
\begin{align*}
g(\zeta)
& = \frac{e^{- \pi/ 2} \zeta}{\pi^2(1-e^{- \pi/2}\zeta)^2} +
\frac{f(\zeta)}{\pi^3} + \frac1{\pi^3} + 
\frac{e^{- \pi/ 2} \zeta^{-1}}{\pi^2(1-e^{- \pi/2}\zeta^{-1})^2} + \frac {f(\zeta^{-1})}{\pi^3} \\ 
& =  \frac{e^{- \pi/ 2} \zeta}{\pi^2(1-e^{- \pi/ 2}\zeta)^2} +
\frac{e^{\pi/ 2} \zeta}{\pi^2(1-e^{\pi/ 2}\zeta)^2} +
\frac{f(\zeta) + 1 + f(\zeta^{-1})}{\pi^3}, 
\end{align*}
which concludes the proof.
\qed
\medskip \\

Now we turn back to the unbounded worm domain $\cW$ via the
biholomorphism $\Phi(z) = (\ell(z),z_2)$,  where $\ell(z) = -i (L(z)
-\log 2)$ and $L(z)$ is given by \eqref{def-L},  and via the
isometric isomorphism
\begin{align*}
T^{-1}: A^2(\cU) &\to A^2(\cW)\\
T^{-1}f(z)& = \frac 1{ iz_1} f\big(\ell(z),z_2\big)\,.
\end{align*}
Recall also that we set $E_\eta(z)=e^{\eta L(z)}$ in \eqref{def-E}.
The next result follows at once from Proposition \ref{cU-kernel}.

\begin{thm} 
The Bergman kernel $K$ of $A^2(\cW)$ can be 
computed at each $(z,w) \in \cW\times \cW$ as
\begin{align}
K(z,w) 
&=( z_1\overline w_1z_2 \overline w_2)^{-1}\sum_{j\in \bbZ}
K_j(\ell(z),\ell(w))\Big( E_{i/2} (z)
z_2\overline{E_{i/2} (w)w_2}\Big)^{j+1} \,.
\end{align}
In particular, when $(z,w) \in
\cW_{\pi/ 2}\times\cW_{\pi/ 2}$, the kernel function takes the form
$$
K(z,w)  =(z_1\overline w_1z_2 \overline w_2)^{-1}\sum_{j\in \bbZ}
K_j\Big(-i\log z_1/2,-i\log w_1/2\Big)
\Big({z_1^{i/2}}z_2\overline{w_1^{i/ 2} w_2} \Big)^{j+1}\, .
$$
\end{thm}

As in the case of $\cU$ we study the boundary behavior of $K$.

\begin{prop}\label{boundarybehavior}
The Bergman kernel $K(z,w)$ of $A^2(\cW)$ extends 
holomorphically in $z$ and antiholomorphically in $w$ near each 
point $(z,w)$ of the boundary of $\cW \times \cW$ except: 
\begin{itemize}
\item[(i)]  when $z_1=0$ or $w_1=0$;
\item[(ii)] when $z_2=0$ or $w_2=0$;
\item[(iii)] when, for some $r \in (0,2]$, 
we have
$$
z_1=r e^{i\log|z_2|^2\pm i\arccos(r/2)}\, ,\quad 
w_1=re^{i\log|w_2|^2\pm i\arccos(r/2)} \quad\text{and}\quad 
\big|\log|z_2|^2 - \log|w_2|^2\big| \leq 2\arccos(r/2)\, .
$$
\end{itemize}
For case (i), we note that there exist a holomorphic function
$H: \cW\times \cW \to \bbC$ with
\begin{align}\label{integrability}
K(z,w) =
\frac{H(z,w)}{z_1\overline w_1z_2 
\overline w_2 \big(\ell(z)-\overline{\ell(w)}\big)^2}
\end{align}
and a holomorphic function $g$ on $A:= \{\zeta : e^{- \pi/2}<|\zeta|<e^{\pi/ 2}\}$
such that:
\begin{itemize}
\item[$(a)$] $H(z,w)$ stays bounded as either $z_1$ or $w_1$ tends to $0$;
\item[$(b)$] if $z_1\to 0$ or $w_1\to 0$ and if 
$E_{i/2} (z)z_2 \overline{E_{i/2}(w) w_2}\to \zeta \in A$ then 
$H(z,w) \to g(\zeta)$;
\item[$(c)$] $g(\zeta) -
[e^{- \pi/ 2} \zeta]/[\pi^2(1-e^{-\pi/ 2}\zeta)^2]
- [e^{\pi/ 2} \zeta]/[\pi^2(1-e^{\pi/ 2}\zeta)^2]$ extends
holomorphically to a neighborhood of $\overline{A}$. 
\end{itemize}
As a consequence, $K$ is singular at all points $(z,w)$ of the
boundary with $z_1=0,z_2 \in \bbC$ or $w_1=0, w_2 \in
\bbC$. 
\end{prop}

\noindent
{\bf Remark.} Case (iii) of Proposition \ref{boundarybehavior} comprises all points $(z,z)$ of the diagonal of $\p\cW\times\p\cW$; but also other points $(z,w)$ of $\p\cW\times\p\cW$, e.g., those such that $z_1= w_1 \in \partial \Delta(0,2)$ and $|z_2|=|w_2|$. See Figure 3 for other cases.
%%%%%%%%%%%%%%%%%%%%%%%%%
\begin{figure}[h]
  \begin{center}
  %\fbox{
  \includegraphics[height=6cm]{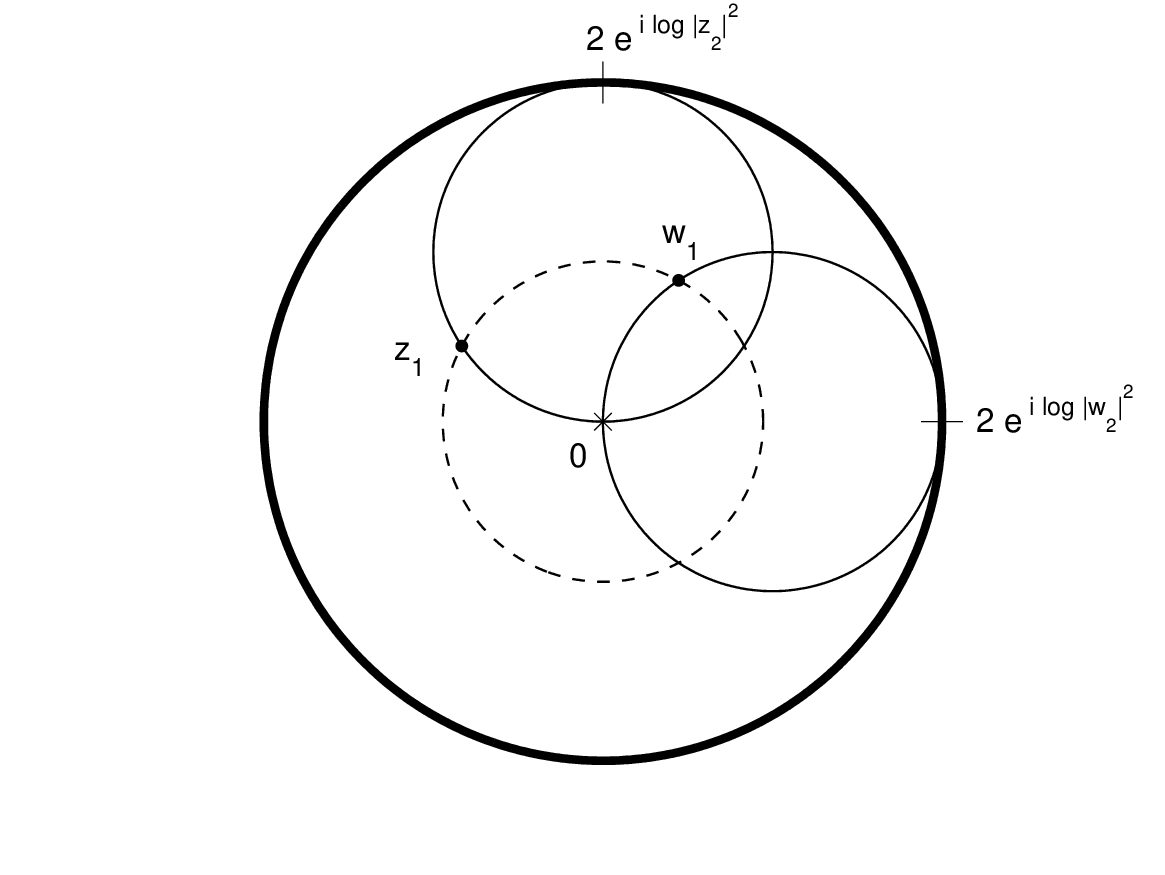}%}
\end{center}
  \caption{\footnotesize Case (iii) of Prop. \ref{boundarybehavior} regards those points $(z,w)$ such that: (1) the points $z_1,w_1 \in \overline{\Delta(0,2)} \setminus \{0\}$ both lie on some circle $\mathscr{C}=\{\zeta \in \bbC: |\zeta|=r\}$ (dashed) and, respectively, on the boundaries $\mathscr{C}_{z_2},\mathscr{C}_{w_2}$ of the discs $\pi_1(\pi_2^{-1}(z_2)), \pi_1(\pi_2^{-1}(w_2))$ (solid); (2) when circling along $\mathscr{C}$ from point $r$ with an orientation such that $z_1$ is the first point of $\mathscr{C}_{z_2}$ encountered, then $w_1$ is the first point of $\mathscr{C}_{w_2}$ encountered; (3) $\big|\log|z_2|^2 - \log|w_2|^2\big| \leq 2\arccos(r/2)$ (which implies that, but is not equivalent to, $w_1 \in \pi_1(\pi_2^{-1}(z_2))$ or $z_1 \in \pi_1(\pi_2^{-1}(w_2))$).
  }
\end{figure}
%%%%%%%%%%%%%%%%%%%%%%%%%

\proof[Proof of Proposition \ref{boundarybehavior}]
The first and second statements are direct consequences of Theorem
\ref{cU-kernel-boundary}, taking into account that $\ell$ extends 
holomorphically to a neighborhood of each point $z$ of 
$\overline{\cW}$ except for those with vanishing $z_1$ or $z_2$.

As for the last statement, we begin by noting that the function 
$z_1\overline w_1z_2 \overline w_2 (\ell(z)-\overline{\ell(w)})^2$ tends to
$0$ as $z_1\overline w_1$ approaches $0$ while $z_2 \overline w_2$
stays bounded; and that $|z_1\overline w_1z_2 \overline w_2|
|\ell(z)-\overline{\ell(w)}|^2$ tends to $+\infty$ as $z_2 \overline w_2 \to
\infty$.

Furthermore, since $g$ extends to a meromorphic function on a 
neighborhood of $\overline{A}$, it can only have finitely many zeros 
in $\overline{A}$. Let $t \in (-\pi/2,\pi/2)$ be such that the circle 
$|\zeta| = e^t$ does not include any zero of $g$. For every $(z,w)$ 
with $z_1=0$ or $w_1=0$, one can easily construct a sequence of 
points tending to $(z,w)$ such that the corresponding values of $H$ 
tend to $g(\zeta)$ with $|\zeta| = e^t$ (hence with $g(\zeta) \neq 0$).
\qed
\medskip

\begin{cor}   \sl  
For $\mu \in
(0,\infty]$ and fixed $w \in \cW$,  the following properties hold:
\begin{itemize}
\item[(1)]
$K(\cdot,w)\not\in
L^p(\cW_\mu)$ for any $p>2$; 
\smallskip
\item[(2)]
$K(\cdot,w)\not\in
W^s(\cW_\mu)$ for any $s>0$. 
\end{itemize}
\end{cor}

\proof
We begin by refining our remarks concerning the function $g$ 
that appears in the previous proposition. As we mentioned in the 
previous proof, $g$ can only have finitely many zeros in 
$\overline{A}$. Fix $w \in \cW$ and set $a:=E_{i/2}(w) w_2$. 
For some $-\pi/ 2<\alpha<\beta< \pi/ 2$, the 
function $z \mapsto \big|g\big(E_{i/ 2}(z)z_2\overline a\big)\big|$ 
is bounded from below by a constant for $z_1$ in the sector
$S(e^{i\log|z_2|^2},\eps) =\{re^{i(t+\log|z_2|^2)}:
\alpha<t<\beta, 0<r<\eps\}$ for all $\eps$ small enough
that $S(e^{i\log|z_2|^2},\eps)\subset\Delta(e^{i\log|z_2|^2},1)$. 

Now, for fixed $\mu \in (0,+\infty)$, let us consider the smooth worm
$\cW_\mu$. We recall that a defining function for $\cW_\mu$ is
$\rho(z) = \big|z_1-e^{i\log|z_2|^2}\big|^2-1+\eta_\mu(\log|z_2|^2) =
|z_1|^2 - 2 {\rm Re}\, (z_1e^{-i\log|z_2|^2}) +\eta_\mu(\log|z_2|^2)$,
where $\eta_\mu$ is an appropriately chosen function such that
 $\eta_\mu^{-1}(0)=[-\mu,\mu]$. As a consequence, $\cW_\mu$ always includes
$\bigcup_{-\mu< \log|z_2|^2<\mu}\Delta(e^{i\log|z_2|^2},1)$. 
Notice that
\begin{align*}
& |\ell(z)- \overline{\ell(w)}|^2 = |L(z)+\overline{L(w)}-2\log2|^2 \\
& = \Big( \log(|z_1|/2)  +\log(|w_1|/2)\Big)^2 +  
\Big( \arg\big(z_1e^{-i\log|z_2|^2} \big)  +\log|z_2|^2 
-\arg\big(w_1e^{-i\log|w_2|^2} \big)  -\log|w_2|^2  \Big)^2 \\
&  \le  \big( \log(|z_1|/2)  +c_1\big)^2 +c_2\,,
\end{align*}
where $c_1 = \log (|w_1|/2)<0$ and $c_2\le (\pi+2\mu)^2$.

Owing to formula \eqref{integrability}, there exist $\eps,C>0$
so that, for all $z \in \bigcup_{-\mu<
  \log|z_2|^2<\mu}S(e^{i\log|z_2|^2},\eps)\times\{z_2\}$, 
\begin{align*}
|K(z,w)| \geq \frac{C}{|z_1| |\ell(z)-\overline{\ell(w)}|^2}  \geq \frac{C}{|z_1|}
 \frac{1}{ \big(\log (|z_1|/2) + c_1\big)^2+c_2} \,.
\end{align*}

Therefore
\begin{align*}
\Vert K(\cdot,w)\Vert_{L^p(\cW_\mu)}^p &\geq \int_{-\mu<
  \log|z_2|^2<\mu} \int_{S(e^{i\log|z_2|^2},\eps)}
\frac{C^p}{|z_1|^p \big[ \big(\log (|z_1|/2) 
+ c_1\big)^2+c_2\big]^{p}} \, dV(z_1)dV(z_2) \\ 
&=\int_{-\mu< \log|z_2|^2<\mu} \int_{S(1,\eps)}
\frac{C^p}{|\zeta|^p \big[ \big(\log (|\zeta|/2) 
+ c_1\big)^2+c_2\big]^{p}} \, dV(\zeta)dV (z_2)\\ 
&=C_\mu
\int_{0}^{\eps} \frac{1}{r^{p-1} \big[\big(\log (r/2) +c_1\big)^2
 + c_2\big]^{p}}\,  dr \, , 
\end{align*}
where the inner integral diverges when $p>2$.

The last statement will be proved for all $s>0$ if we can prove it for
all $s \in (0, \frac 1 2)$. In the latter case, according to
\cite{Li}, the function $K(\cdot,w)$ belongs to the Sobolev space
$W^s(\cW_\mu)$ if and only if $\rho(\cdot)^{-s}K(\cdot,w)$ is in
$L^2(\cW_\mu)$. But 
\begin{align*}
&\Vert \rho(\cdot)^{-s}K(\cdot,w)\Vert_{L^2(\cW_\mu)}^2\\
&\geq \int_{-\mu< \log|z_2|^2<\mu}
\int_{S(e^{i\log|z_2|^2},\eps)}  \frac{C^2}{\big||z_1|^2 - 2 \Re
  (z_1e^{-i\log|z_2|^2})\big|^{s}|z_1|^2 \big[ \big(\log (|z_1|/2) +
  c_1\big)^2+c_2\big]^2}
\, dV(z_1)dV (z_2) \\ 
&=\int_{-\mu< \log|z_2|^2<\mu} \int_{S(1,\eps)}  \frac{C^2}{\big||\zeta|^2 - 2\Re
  (\zeta)\big|^{s}|\zeta|^2 \big[ \big(\log (|\zeta|/2) +
  c_1\big)^2+c_2\big]^2}\,
dV(\zeta)dV (z_2)\\ 
&= C\int_{\alpha}^{\beta} \int_{0}^{\eps}  \frac{1}{\big|r
  -2 \cos t\big|^{s} r^{1+s} \big[\big(\log (r/2) +c_1\big)^2
 + c_2\big]^2}\,  dr dt 
\end{align*}
where the inner integral diverges when $s>0$, for all $t \in (\alpha,\beta)$.
\qed
\medskip \\

\proof[Proof of Theorem \ref{irregularity}]    
We saw in the previous theorem that $K_w = K(\cdot,w)$ does not belong
to $W^s(\cW)$ nor to $L^p(\cW)$ for any $s>0$ or $p>2$. Since $K_w$ can be
obtained as the projection $\cP(\chi_w)$ of a smooth cut-off
function $\chi_w \in C^\infty_0$ supported in a compact neighborhood
of $w$ (see \cite{Ker}), the inclusion $\cP(W^s(\cW)) \subseteq W^s(\cW)$
implies $s\leq0$ and $\cP(L^p(\cW)) \subseteq L^p(\cW)$ implies
$p\leq 2$. 

We complete the proof by showing that $\cP(L^p(\cW)) \subseteq
L^p(\cW)$ implies $p\geq 2$. This part of the proof makes use of the duality 
between $L^p$ and $L^{p'}$ with $\frac{1}{p} + \frac{1}{p'} =1$. We observe 
that, since $\cP f(w) = \la f,K_w\ra$, 
\begin{align*}
\Vert K_w\Vert_{L^{p'}} &=\sup_{\Vert f\Vert_{L^p}=1} \big|
  \int_{\cW} f(z)\overline{K_w(z)}\, dV(z) \big|  =
\sup_{\Vert f\Vert_{L^p}=1} |\cP f(w)|\\ 
&\leq \sup_{\Vert f\Vert_{L^p}=1} \big|\frac 1 {V(B)}
  \int_B \cP f(z)\, dV(z)\big| \leq C \sup_{\Vert
  f\Vert_{L^p}=1} \Vert \cP f \Vert_{L^p} \leq C', 
\end{align*}
which implies $p' \leq 2$, hence that $p\geq2$ as desired.
\qed

\section{Concluding Remarks}

We have studied the worm now for several years and met with some
success in analyzing the unbounded (sometimes non-smooth)
worm. See for instance \cite{KrPe2},
\cite{KrPe3}, \cite{KrPe}. Our ultimate goal, however, is to study the
original worm domain $\cW_\mu$ of Diederich and Forn\ae ss
\cite{DFo1}.

The approach used in the present paper allows, even in the case of
 $\cW_\mu$, to reduce the study of the Bergman space of to a family 
 of weighted Bergman spaces on a planar domain. In this case the 
 planar domain is not a half-plane anymore and the weight depends 
 on both real variables, two facts which prevent from 
 computing the kernel with the technique used for $\cW$. However, 
 the reduction to a planar domain may shed some light on the 
 challenging problem of writing down a complete system for the 
 Bergman space of $\cW_\mu$. We intend to explore these matters 
 in a forthcoming paper.

We also intend to apply the approach used in the present paper to the 
higher-dimensional version of the worm domain introduced and studied 
by Barrett and S.~\c{S}ahuto\u{g}lu in \cite{BaSa}. Namely, for $n\ge3$ 
they defined the domain
\begin{equation}\label{higher-dim-worm}
\Omega_{\alpha\beta} 
= \big\{ (z_1, z', z_n) \in \bbC^n : \, r(z_1, z', z_n)< 0\big\}
\end{equation}
where
$$
r(z_1, z', z_n) = \big|z_1-e^{i\alpha\log|z_n|^2}\big|^2 +|z'|^2
-1+\sigma(|z_n|^2-\beta)+\sigma(1-|z_n|^2) \, ,
$$
$z_1,z_n\in \bbC$, $z'\in \bbC^{n-2}$, $\alpha>0$, $\beta>1$ and
$\sigma(t)=M\chi_{(0,+\infty)}(t) e^{-1/t}$, for some $M>0$.  They
proved that the Bergman projection on $\Omega_{\alpha\beta}$ is
irregular on the Sobolev space $W^{s,p}(\Omega_{\alpha\beta})$  
when $1\le p<\infty$ and
$s\ge \frac{\pi}{2\alpha\log\beta}+ n\big(\frac1p-\frac12\big)$.
Here  $W^{s,p}(\Omega_{\alpha\beta})$  denotes the space of 
functions whose derivatives up to order $s$ are $L^p$-integrable.
In particular, our approach may apply to study the unbounded 
domain obtained from $\Omega_{\alpha\beta}$ by letting 
$\beta\to+\infty$.
\bigskip \bigskip \\

\bibliography{worm-mathscinet}
\bibliographystyle{amsalpha}
\vskip12pt

\end{document}